\documentclass[11pt]{amsart}
\usepackage{amsmath,amssymb,amscd,amsthm,amsxtra}
\usepackage{latexsym}
\usepackage{graphicx}
\usepackage{bbm}
\usepackage[colorlinks=true, pdfstartview=FitV, linkcolor=blue, citecolor=blue, urlcolor=blue]{hyperref}
\usepackage{color}

\usepackage{graphicx,epstopdf,verbatim}

\allowdisplaybreaks

\newtheorem{theorem}{Theorem}[section]
\newtheorem{lemma}[theorem]{Lemma}
\newtheorem{proposition}[theorem]{Proposition}
\newtheorem{definition}[theorem]{Definition}
\newtheorem{corollary}[theorem]{Corollary}

\theoremstyle{remark}
\newtheorem{remark}[theorem]{Remark}

\newcommand{\R}{\mathbb{R}}

\def\nn{{\mathbb N}}
\numberwithin{equation}{section}

\begin{document}
\title[Two weight  estimates for dyadic operators]{\textbf{On  two weight   estimates for dyadic operators}}

\author[O. Beznosova, D. Chung, J.C. Moraes, and M.C. Pereyra]{Oleksandra Beznosova, Daewon Chung, Jean Carlo Moraes, and Mar\'{i}a Cristina Pereyra}

\address{Oleksandra Beznosova\\
Department of Mathematics\\
University of Alabama\\ 345 Gordon Palmer Hall\\ Tuscaloosa, AL 35487}
\email{ovbeznosova@ua.edu}

\address{Daewon Chung\\
Faculty of Basic Sciences, Mathematics\\
Keimyung University\\1095 Dalgubeol-daero\\
 Daegu 42601, Korea}
 \email{dwchung@kmu.ac.kr}
\thanks{The first author was supported by the University of Alabama RGC grant.}
\thanks{The second author was supported by Basic Science Research Program through the National
Research Foundation of Korea(NRF) funded by the Ministry of Science, ICT $\&$ Future
Planning(2015R1C1A1A02037331).}
\thanks{The third author was supported by the AVG program, 459895/2013-3, funded by Conselho Nacional de Desenvolvimento Cient\'ifico e Tecnol\'ogico, CNPq.}
\address{Jean Carlo Moraes\\
Instituto de Matem\'atica\\
Universidade Federal do Rio Grande do Sul\\
Av. Bento Goncalves, 9500 - Pr\'edio 43-111\\
Agronomia, Caixa Postal 15080, 91509-900, Porto Alegre, RS, Brazil}
\email{jean.moraes@ufrgs.br}

\address{Mar\'{i}a Cristina Pereyra\\
Department of Mathematics and Statistics\\
MSC03 21501\\ University of New Mexico\\
Albuquerque, NM 87131-0001} \email{crisp@math.unm.edu}

\dedicatory{In memory of our good  friend and mentor Cora Sadosky}

\begin{abstract}
We provide a quantitative  two weight estimate for the dyadic
  paraproduct  $\pi_b$ under certain conditions on a pair of weights $(u,v)$ and
  $b$ in  $Carl_{u,v}$, a new class of functions that we show coincides with $BMO$
  when $u=v\in A_2^d$.
  We obtain quantitative two weight estimates for the dyadic square function and the martingale transforms
  under the assumption that the maximal function is bounded from $L^2(u)$ into $L^2(v)$ and $v\in RH^d_1$.
  Finally we obtain a  quantitative two weight estimate from $L^2(u)$ into $L^2(v)$ for the dyadic square function
  under the assumption that the pair $(u,v)$ is in joint $\mathcal{A}_2^d$ and $u^{-1}\in RH_1^d$,
  this is sharp in the sense that when $u=v$ the conditions reduce to $u\in A^d_2$ and the estimate is the known
  linear mixed estimate.
\end{abstract}


\maketitle

\section{Introduction}


We study quantitative two weight inequalities for some dyadic operators. More precisely, we study conditions on pairs of
locally integrable a.e. positive functions $(u,v)$ so that a linear or sublinear dyadic operator  $T$ is bounded from $L^2(u)$ into $L^2(v)$,
that is there exists a constant $C_{T,u,v}>0$ such that for all $f\in L^2(u)$,
\[ \|Tf\|_{L^2(v)}\leq C_{T,u,v} \|f\|_{L^2(u)},\]
with estimates on $C_{T,u,v}$ involving the constants that appear in the conditions imposed on the weights and/or the operator.

There are two current schools of thought regarding the two weight problem. First, given one operator find necessary and sufficient conditions
on the weights to ensure boundedness of the operator on the appropriate spaces. Second, given a family of operators
find necessary and sufficient conditions on the weights to ensure boundedness of the family of operators.
In the first case, the conditions are usually ``testing conditions" obtained from checking boundedness of the given operator on
a collection of test functions. In the second case, the conditions are more ``geometric", meaning  to only involve the weights and
not the operators, such as Carleson conditions or bilinear embedding conditions, Muckenhoupt $A_2$ type conditions or bumped conditions.
Operators of interest are the maximal function \cite{S1, Moe, PzR, V}, fractional and Poisson integrals \cite{S2,Cr}, the Hilbert transform \cite{CS1, CS2, KP, NTV1, LSSU, L3}
and general Calder\'on--Zygmund singular integral operators and their commutators \cite{CrRV, CrMoe,CrMPz2,  NRTV}, 
the square functions \cite{ LLi,LLi2,CLiX,HLi}, paraproducts and their dyadic counterparts \cite{M, HoLWic1, HoLWic2,IKP,Be3}. 
Necessary and sufficient conditions are only known for the maximal function, fractional and Poisson integrals \cite{S1}, 
square functions \cite{LLi} and the Hilbert transform \cite{L3,LSSU}, and among the dyadic operators for the martingale 
transform, the dyadic square functions, positive and well localized dyadic operators 
 \cite{Wil87, NTV1,NTV3, T, Ha, HaHLi, HL, LSU2, Ta, Vu1,Vu2}. If the weights $u$ and $v$ are assumed to be in $A^d_2$, 
then necessary and sufficient conditions for boundedness of  dyadic paraproducts and commutators in terms of
 Bloom's $BMO$ are known \cite{HoLWic1,HoLWic2}. The assumption that a weight is in dyadic $A^d_p$ is a 
strong assumption, it implies, for example, that the  weight is dyadic doubling. On the other hand
if the paraproduct is adapted to the weights $u$ and $v$, then necessary and sufficient conditions
for its boundedness from $L^p(u)$ into $L^p(v)$  are known \cite{LaT} even in  the non-homogeneous case, interestingly enough the conditions are different depending on whether
$1\leq p <2$ or $p>2$.

In this paper we obtain a quantitative two weight estimate for $\pi_b$, the dyadic paraproduct associated to $b$, where $b\in Carl_{u,v}$
a new class of functions that we show coincides with $BMO^d$ when
$u=v\in A_2^d$.
The sufficient conditions on the pair of weights $(u,v)$ required in our theorem are half of the conditions required for
the  boundedness of the martingale transform, namely (i) $(u,v)\in \mathcal{A}_2^d$ (joint dyadic $A_2$ condition) and (ii) a Carleson condition on the weights,
or equivalently, the conditions required for the boundedness of the dyadic square function from $L^2(v^{-1})$ into $L^2(u^{-1})$.

In what follows $\mathcal{D}$ denotes the dyadic intervals, $\mathcal{D}(J)$ denotes the dyadic subintervals of an interval $J$, $|J|$ denotes the length of the interval $J$,
$\{h_I\}_{I\in \mathcal{D}}$ denotes the Haar functions, $m_If:= \frac{1}{|I|}\int_I f$ denotes the integral average of $f$ over the interval $I$ with respect to Lebesgue measure, and $\langle f, g\rangle := \int f \overline{g}$
denotes the inner product on $L^2(\R )$. We prove the following theorem.


\begin{theorem}\label{T:para1} Let $(u,v)$ be a pair of measurable functions on $\R$ such that $v$ and
$u^{-1}$, the reciprocal of $u$, are 
weights  on $\R$, and such that

\emph{(i)} $(u,v)\in \mathcal{A}_2^d$, that is
$ \; [u,v]_{\mathcal{A}_2^d} := \sup_{I\in\mathcal{D}}m_I (u^{-1})\, m_I v <\infty.$

\emph{(ii)} 
there is a constant $\mathcal{D}_{u,v}>0$ such that
\[ \sum_{I\in\mathcal{D}(J)} |\Delta_I v|^2|I|\, m_I(u^{-1}) \leq \mathcal{D}_{u,v} v(J) \quad\quad \mbox{for all} \; J\in\mathcal{D},\]
where $\Delta_I v :=  m_{I_+}v-m_{I_-}v$, and $I_{\pm}$ are the right and left children of $I$.

Assume that $b\in Carl_{u,v}$, that is $b\in L^1_{loc}(\R)$ and
there is a constant $\mathcal{B}_{u,v}>0$ such that
\[ \sum_{I\in\mathcal{D}(J)} \frac{|\langle b, h_I\rangle|^2}{m_Iv} \leq \mathcal{B}_{u,v} u^{-1}(J)\quad\quad \mbox{for all} \; J\in\mathcal{D}.\]

Then $\pi_b$, the dyadic paraproduct associated to $b$, is bounded from
$L^2(u)$ into $L^2(v)$. 
Moreover,  there exists a constant  $C>0$ such that for all $f\in L^2(u)$
$$\|\pi_bf\|_{L^2(v)}\leq C\sqrt{[u,v]_{\mathcal{A}_2^d}\mathcal{B}_{u,v}}
\Big(\sqrt{[u,v]_{\mathcal{A}_2^d}}+\sqrt{\mathcal{D}_{u,v}}\,\Big)\|f\|_{L^2(u)}\,,$$
where $\pi_bf := \sum_{I\in\mathcal{D}} m_If\,\langle b,h_I\rangle\, h_I$.
\end{theorem}

When $u=v=w$ the conditions in Theorem~\ref{T:para1} reduce to $w\in A_2^d$ and $b\in BMO^d$, but we do not 
recover the first author's linear bound for
the dyadic paraproduct \cite{Be1},  we are off by a factor of $[w]^{1/2}_{RH_1^d}$.
 In \cite{M, MP} similar methods yield the
linear bound in the one weight case, but there is a step in that argument that can not be taken in the two weight setting.
 More precisely, in the one weight case, $u=v=w$, we have $ww^{-1}=1$ and $1\leq m_Iw\,m_I(w^{-1})$; in the two weight
 case we can no longer bound $vu^{-1}$  nor can  we bound  $m_Iv\,m_I(u^{-1})$ positively away from zero.

We  compare the known two weight results for the martingale transform, the dyadic square function,
and the dyadic maximal function. Assuming the
maximal  operator is bounded  from $L^2(u)$ into $L^2(v)$, and under the additional condition that
$v$ is in the $RH_1^d$ class, we conclude the other operators are bounded
with quantitative estimates involving the operator norm of the maximal function and the $RH^d_1$ constant.
Notice that the boundedness of the maximal function implies
that the weights $(u,v)$ obey the joint $\mathcal{A}_2^d$ condition, but this is not sufficient for boundedness neither of the martingale transform nor the dyadic square function.
Finally we obtain quantitative two weight  estimates for the dyadic square function when  $(u,v)\in \mathcal{A}_2^d$ and $u^{-1}$ is in $RH^d_1$. This extends
work of the first author \cite{Be2} where similar quantitative two weight bounds were obtained under the stronger assumption that $u^{-1}\in A_q^d$ for some $q>1$ (in other words,
$u^{-1}\in A^d_{\infty}$).

\begin{theorem}\label{T:square1} Let $(u,v)$ be a pair of  measurable functions such that $(u,v)\in \mathcal{A}_2^d$ and
$u^{-1}\in RH_1^d\,.$ Then there is a constant such that
$$\|S^d\|_{L^2(u)\rightarrow L^2(v)}
\leq C[u,v]_{\mathcal{A}^d_{2}}^{1/2}\big (1+[u^{-1}]_{RH_1^d}^{1/2}\big )\,.$$
\end{theorem}

When the two weights equal $w$ the conditions in Theorem~\ref{T:square1} reduce to $w\in A_2^d$ and we improve the sharp linear estimates of Hukovic
et al \cite{HTV} to a mixed linear estimate.  Compare to  one weight  mixed type estimates of
Lerner \cite{Le2}, and two weight strong and weak estimates in \cite{CLiX, LLi, HLi} where similar estimates are obtained for the $g$-function and
Wilson's intrinsic square function \cite{Wilbook}. In the aformentioned papers, both weights are assumed to be in $A_{\infty}$.

 The one weight  problem, corresponding to $u=v=w$ is well understood. In 1960,  Helson and Szeg\"{o} (\cite{HS}) presented the first necessary and sufficient conditions on $w$ for the boundedness of the Hilbert transform on $L^2(w)$
in the context of prediction theory. They used methods involving analytic functions and operator theory. The two weight
characterization for the Hilbert transform in this direction was completely solved by Cotlar and Sadosky in \cite{CS1}
and \cite{CS2}. The class of $A_p$ weights was introduced in 1972 by Muckenhoupt (\cite{Mu}),  these are the weights $w$
for which the Hardy-Littlewood maximal function maps $L^p(w)$ into itself. We say the positive almost
everywhere and locally integrable function $w$ satisfies the $A_p$ condition if and only if
$$[w]_{A_p}:=\sup_I \bigg(\frac{1}{|I|}\int_Iw(x)dx\bigg)
\bigg(\frac{1}{|I|}\int_Iw^{-\frac{1}{p-1}}(x)dx\bigg)^{p-1}<\infty,$$
where $[w]_{A_p}$ denotes the $A_p$ characteristic (often called $A_p$ constant or norm) of
the weight.  In 1973, Hunt, Muckenhoupt, and Wheeden \cite{HMW} showed that the Hilbert transform
is bounded on $L^p(w)$ if and only if $w\in A_p\,.$ Also, in 1973,  Coiffman and
Fefferman~\cite{CoFe} extended this result to the classical Calder\'{o}n-Zygmund operators.
When $u=v=w$ the  joint $\mathcal{A}_2$ condition coincides with $A_2$.  The joint $\mathcal{A}_2$ condition
is necessary and sufficient for the two weight weak (1,1) boundedness of the maximal function but is not enough for the strong
 boundedness \cite{S1}. 
 In 1982 Sawyer found necessary and sufficient conditions on pairs of weights for the boundedness
of the maximal function, namely joint $\mathcal{A}_2$ and the testing conditions \cite{S1}.
In the 90's the interest shifted toward the study, in the one weight case, of
the sharp dependence of $A_p$ characteristic for a general Calder\'{o}n-Zygmund operator on
weighted Lebesgue spaces $L^p(w)$. In 2012 Hyt\"onen proved the $A_2$-conjecture (now theorem):
Let $T$ be a Calder\'{o}n-Zgmund operator and
$w$ be an $A_2$ weight then
\[\|Tf\|_{L^2(w)}\leq C\,[w]_{A_2}\|f\|_{L^2(w)}\,,\]
where the constant $C$ depends only on the dimension $d$, the growth and smoothness of the
kernel of $T$, and its norm in the non-weighted $L^2$.
From sharp extrapolation \cite{DGPPet} one deduces that  for $1<p<\infty\,,$ and $w\in A_p$,
$$\|Tf\|_{L^p(w)}\leq C_{d,T,p}\,[w]^{\max\{1,1/(p-1)\}}_{A_p}\|f\|_{L^p(w)}\,.$$
After these groundbreaking results,  improvements were found in the form of
mixed type estimates such as the following  $L^2(w)$ estimate
\[ \|Tf\|_{L^2(w)}\leq C\,[w]^{1/2}_{A_2}([w]_{A_{\infty}^d}^{1/2} + [w^{-1}]_{A_{\infty}^d}^{1/2})\|f\|_{L^2(w)},\]
where $A_{\infty}^d= \cup_{p>1} A_p^d$ , and $[w]_{A_{\infty}^d}$  is the Hru\v{s}\v{c}ev constant or is replaced by
the  smaller   $[w]_{RH_1^d}$ as we do in this paper,  see \cite{HL,HP} and \cite{LeMoe,PzR} for other variations.
Currently a lot of effort has been put into finding two weight analogues of these estimates
as described at the beginning of this introduction. In this paper we present two weight quantitative
and mixed type estimates for the dyadic paraproduct, martingale transform, and the dyadic square function.

In this paper we work in $\R$ but the results should hold in $\R^d$ and in spaces of homogeneous type.

Definitions and frequently used theorems are collected in Section 2, including joint $\mathcal{A}_2^d$,
regular and weighted Haar functions, $w$-Carleson sequences,
the class $Carl_{u,v}$, the class $RH_1^d$ and its quantitative relation to $A_{\infty}^d$, weighted
Carleson's and Buckley's Lemmas. The main dyadic operators are introduced
in Section 3: dyadic maximal function, dyadic square function, martingale transform and the dyadic
paraproduct, we record the known two weight results for these operators. In Section 4 we
prove our quantitative two weight result for the dyadic paraproduct, we also show that when $u=v \in A_2^d$ then $Carl_{u,u}=BMO^d$.
We compare our conditions to bumped conditions and argue that neither result implies the other, we also compare $Carl_{u,v}$ 
to the Bloom $BMO$ and related conditions. 
 In Section 5
we obtain some quantitative two weight estimates for the dyadic square function and the martingale transforms
under the assumptions that the maximal function is bounded and the additional assumption  $v$ is a weight in $RH_1^d$.
In Section 6 we obtain a sharp two weight estimate for the dyadic square function under the assumptions that
$(u,v)\in \mathcal{A}_2$  and $u^{-1}\in RH_1^d$.

The authors would like to thank the referee for thoughtful comments, and for enticing us to explore in more depth
the Bloom BMO condition and compare it to  $Carl_{u,v}$. The authors would also like to thank Jethro van Ekeren, a friend of the third author and a native English speaker,  who proofread  the article.
\section{Definitions and frequently used theorems.}\label{Defandlemma}

Throughout the proofs a constant $C$ will be a numerical constant that
may change from line to line.  The symbol $A_n\lesssim B_n$ means there is a constant $c>0$ independent of $n$ such that $A_n \leq cB_n$, and $A_n \approx B_n$ means that $A_n\lesssim B_n$
and $B_n\lesssim A_n$.
Given a measurable set $E$ in $\R$, $|E|$ will denote its Lebesgue
measure. We say that  a function $v:\R\to\R$ is a weight if $v$ is an almost everywhere positive locally integrable function.
For a given weight $v$, the $v$-measure of a measurable set $E$,
denoted by $v(E)$, is $v(E)=\int_Ev(x)dx$.
We say that a weight $v$ is a regular weight if
$v((-\infty, 0))= v ((0,\infty)) = \infty$. Let us denote $\mathcal{D}$ the collection of
all dyadic intervals, and let us denote $\mathcal{D}(J)$ the collection of all dyadic
subintervals of $J\,.$

We say that a pair of weights $(u,v)$ satisfies the joint $\mathcal{A}_2^d$ condition if and only if both $v$ and $u^{-1}$, the reciprocal of $u$, are weights, and
\begin{equation}[u,v]_{\mathcal{A}^d_2}:=\sup_{I \in \mathcal{D}} m_I (u^{-1}) \,m_I v < \infty,
\label{JointA2}\end{equation}

\noindent where $m_Iv$ stands for the integral average of a weight $v$ over the interval $I\,.$
Note that $(u,v)\in \mathcal{A}^d_2$ is equivalent to $(v^{-1},u^{-1})\in \mathcal{A}^d_2$ and the corresponding constants are equal.
Similarly a pair of weights $(u,v)$ satisfies the joint $\mathcal{A}_p^d$ condition iff
\[[u,v]_{\mathcal{A}^d_p}:=\sup_{I \in \mathcal{D}} m_I (u^{\frac{-1}{p-1}})^{p-1} m_I v < \infty. \]
Note also that $(v,v)\in \mathcal{A}_p^d$ coincides with the usual one weight  definition of $v \in A_p^d$.
\subsection{Haar bases}
For any interval
$I\in\mathcal{D}$, there is a Haar function defined by
$$h_I(x)=\frac{1}{\sqrt{|I|}}\Big(\mathbbm{1}_{I_+}(x)-\mathbbm{1}_{I_-}(x)\Big)\,,$$

\noindent where $\mathbbm{1}_I$ denotes the characteristic function of the interval $I\,$,
and $I_+$, $I_-$ denote the right and left child of $I$ respectively.
For a given weight $v$ and an interval $I$ define the weighted Haar function as
$$h_I^v(x)=\frac{1}{\sqrt{v(I)}}\left(\sqrt{\frac{v(I_-)}{v(I_+)}}
\mathbbm{1}_{I_+(x)}-\sqrt{\frac{v(I_+)}{v(I_-)}}\mathbbm{1}_{I_-(x)}\right)\,.$$

\noindent The space  $L^2(v)$ is the collection of square integrable complex valued functions with respect
to the measure $d\mu=vdx$,  it is a Hilbert space with the weighted inner product
defined by $\langle f,g\rangle_v=\int f\overline{g}vdx$.
It is a well known fact that the Haar systems $\{h_I\}_{I\in\mathcal{D}}$
and $\{h^v_I\}_{I\in\mathcal{D}}$ are orthonormal systems in $L^2(\R )$ and $L^2(v)$
respectively.
Therefore, for any weight $v$, by  Bessel's inequality we have the
following:
$$\sum_{I\in\mathcal{D}}|\langle f,h_I^v\rangle_v|^2\leq \|f\|_{L^2(v)}^2\,.$$

\noindent  Furthermore, if $v$ is a regular weight, then every function $f\in L^2(v)$
can be written as
$$f=\sum_{I\in\mathcal{D}}\langle f,h^v_I\rangle_vh^v_I\,,$$

\noindent where the sum converges a.e. in $L^2(v),$ hence the  family
$\{h_I^v\}_{I\in\mathcal{D}}$ is a complete orthonormal system. Note that if $v$ is
not a regular weight so that $v((-\infty,0))$, $v((0,\infty))$, or both are finite,
then  either $\mathbbm{1}_{(-\infty,0)}$, $\mathbbm{1}_{(0,\infty)}$, or both are in $L^2(v)$ and are orthogonal to $h_I^v$ for
every dyadic interval $I$.

The weighted and unweighted Haar functions
are related linearly as follows:

\begin{proposition}\emph{\cite{NTV1}}\label{WHB}
For any weight $v$ and every $I\in\mathcal{D} $, there are numbers $\alpha_I^v$,
$\beta^v_I$ such that
$$ h_I(x) = \alpha^v_I \,h^v_I(x) + \beta_I^v \,\frac{\mathbbm{1}_I(x)}{\sqrt{|I|}}$$

\noindent where \emph{(i)} $|\alpha^v_I | \leq \sqrt{m_Iv},$
\emph{ (ii)}  $|\beta^v_I| \leq \frac{|\Delta_I v|}{m_Iv},$
and $\Delta_I v:= m_{I_+}v - m_{I_-}v.$
\end{proposition}

\subsection{Dyadic $BMO$}
A locally integrable function $b$ is in the space of dyadic bounded mean oscillation ($BMO^d$)
if and only if there is a constant $C>0$ such that for all $I\in\mathcal{D}$ one has
\[  \int_I |b(x)-m_Ib|\,dx\leq C|I|.\] 
The smallest constant $C$ is the $BMO^d$-norm of $b$.
The celebrated John-Nirenberg Theorem (see \cite{P1}) implies that for each $1\leq p<\infty$,  $b\in BMO^d$ iff
 \[ \|b\|^p_{BMO^d_p}:= \sup_{I\in\mathcal{D}}\frac{1}{|I|}\int_I |b(x)-m_Ib|^pdx <\infty.\]
Furthermore $\|b\|_{BMO^d_p}$ is comparable to the $BMO$-norm of $b$.

In this paper we will mostly be concerned with $p=2$ and we will declare 
\[\|b\|_{BMO^d}:=\|b\|_{BMO^d_2}=\sup_{I\in\mathcal{D}}\Big (\frac{1}{|I|}\int_I |b(x)-m_Ib|^2dx\Big )^{1/2}.\]

\begin{lemma}\label{lem:BMO2norm} If $b\in BMO^d$ then
\[ \|b\|_{BMO^d}^2 = \sup_{I\in\mathcal{D}} \frac{1}{|I|}\sum_{J\in\mathcal{D}(I)} |\langle b,h_J\rangle|^2.\]
\end{lemma}
\begin{proof} The family $\{h_J\}_{J\in\mathcal{D}(I)}$ is an orthonormal basis of the space 
$L^2_0(I):=\{f\in L^2(I): \int_If=0\}.$ The function $(b-m_Ib)\mathbbm{1}_I\in L^2_0(I)$, hence by Plancherel
\[ \int_I |b(x)-m_Ib|^2dx= \sum_{J\in\mathcal{D}(I)} |\langle b,h_J\rangle|^2.\]
This proves the lemma.
\end{proof}

In other words, $b\in BMO^d$ if and only if there is a constant $C>0$ such that for all $I\in\mathcal{D}$
\[ \sum_{J\in\mathcal{D}(I)} |\langle b,h_J\rangle|^2\leq C |I|.\]

\subsection{Carleson sequences} \label{S:Def_Carl}
A positive sequence $\{\lambda_I\}_{I\in\mathcal{D}}$ is a $v$-Carleson sequence if
there is a constant $C>0$ such that for all dyadic intervals $J$
\begin{equation}
\sum_{I\in\mathcal{D}(J)}\lambda_I\leq Cv(J)\,.\label{CarlS}\end{equation}

\noindent When $v=1$ almost everywhere we say that the sequence is a
Carleson sequence or a $dx$-Carleson sequence. The infimum among all $C$'s that satisfy the inequality (\ref{CarlS})
is called the intensity of the $v$-Carleson sequence $\{\lambda_I\}_{I\in\mathcal{D}}\,.$
For instance, $b\in BMO^d$ if and only if $\{|\langle b,h_I\rangle|^2\}_{I\in\mathcal{D}}$
is a Carleson sequence with intensity $\|b\|^2_{BMO^d}$.
The following lemma gives a relationship between unweighted and weighted Carleson sequences.

\begin{lemma}[Little Lemma, \emph{\cite{Be1}}]  \label{L:little} Let $v$ be a weight, such that $v^{-1}$ is also a
weight, and let $\{\alpha_I\}_{I\in\mathcal{D}}$ be a Carleson sequence with
intensity $B$ then $\{\alpha_I/m_I(v^{-1})\}_{I\in\mathcal{D}}$ is a $v$-Carleson
sequence with intensity at most $4B$, that is for all $J\in\mathcal{D}$,
$$\frac{1}{|J|}\sum_{I\in\mathcal{D}(J)}\frac{\alpha_I}{m_I(v^{-1})}\leq 4Bm_Jv\,.$$
\end{lemma}


We also need to define a class of objects that will take the place of the $BMO^d$ class
in the two weighted case, we will call this class the two weight Carleson class.

\begin{definition}\label{TWCC} Given a pair of functions $(u,v)$  such that $v$  and $u^{-1}$
are 
weights, we say that a locally integrable
function $b$ belongs to the two weight Carleson class, $Carl_{u,v}$, if
$\big\{|b_I|^2/m_Iv\}_{I\in\mathcal{D}}$ is a $u^{-1}$- Carleson sequence where
 $b_I=\langle b,h_I\rangle\,.$
\end{definition}

Note that if $u=v$, then we have that $b\in Carl_{v,v}$ iff $\{|b_I|^2/m_Iv\}_{I\in \mathcal{D}}$
is a $v^{-1}$-Carleson sequence. The later statement  is true  if
$\{|b_I|^2\}_{I\in\mathcal{D}}$ is a Carleson sequence (by Lemma~\ref{L:little}),
which in turn is equivalent
to saying that $b\in BMO^d$. Therefore for any weight $v\,,$ such that $v^{-1}$ is
also a weight, we have that $$BMO^d\subset Carl_{v,v}\,.$$
Moreover, if $\mathcal{B}_{v,v}$ is the intensity of the $v^{-1}$-Carleson sequence
$\{|b_I|^2/m_Iv\}_{I\in\mathcal{D}}$ then $\mathcal{B}_{v,v}\leq 4\|b\|_{BMO^d}^2\,$.
In Section~\ref{Paraproduct} we will show that if $v\in A_2^d$ then $BMO^d = Carl_{v,v} \cap L^2_{loc}(\R)$ (see Corollary~\ref{C:BMOCvv}).\\


We now introduce some useful lemmas which will be used frequently throughout  this paper.
You can find proofs in \cite{MP}. The  following lemma was
stated first in \cite{NTV1}.

\begin{lemma}[Weighted Carleson Lemma] \label{WCL}
Let $v$ be a  weight, then $\{\alpha_{I}\}_{I \in \mathcal{D}}$ is a $v$-Carleson sequence
with intensity $\mathcal{B}$ if and only if for all non-negative $v$-measurable functions  $F$ on the
line,
\begin{equation}\label{eqn:WCL}
\sum_{I \in \mathcal{D}} (\inf_{x \in I} F(x) )\alpha_{I} \leq \mathcal{B}
\int_{\mathbb{R}}F(x) \,v(x)\,dx.
\end{equation}
\end{lemma}

In relation to Carleson sequences we consider another class of weights which is called the Reverse H\"{o}lder class
with index 1 and is defined as follows.

\begin{definition}
A weight $v$ belongs to the dyadic Reverse H\"{o}lder class $RH_1^d$ whenever its characteristic $[v]_{RH_1^d}$ is finite,
where
$$[v]_{RH_1^d}:=\sup_{I\in\mathcal{D}}m_I\bigg(\frac{v}{m_Iv}\log\frac{v}{m_Iv}\bigg)<\infty\,.$$
\end{definition}

It is  well known that $v\in A_{\infty}$ if and only if $v\in RH_1$.
In the dyadic case, $v\in RH_1^d$ does not  imply that $v$ is dyadic doubling, however
$v\in A_{\infty}^d$ does.  See \cite{P1} for more details. Recently, the first author and Reznikov
obtained, in \cite{BeRe},  the sharp comparability of the $A^d_{\infty}$ and $RH_1^d$ characteristics.

\begin{theorem}\emph{\cite{BeRe}}. If a weight $v$ belongs to the $A^d_{\infty}$ class,
then $v\in RH_1^d$. Moreover,
$$[v]_{RH_1^d}\leq \log(16)[v]_{A_{\infty}^d}.$$

\noindent The constant $\log (16)$ is the best possible.
\end{theorem}


We would also like to note here that results of Iwaniek and Verde \cite{IwVe} show that 
 $[w]_{RH^d_1} \approx \sup_{I \in D}\frac{\|w\|_{L\log L, I}}{\|w\|_{L,I}}$,
where $\|\cdot\|_{\Phi(L),I}$ stands for the $\Phi(L)$-Luxemburg norm (for more details  see \cite{BeRe}).
In the same paper you can find the following characterization of the $L\log L$-norm (Part (a)) and a sharp version of Buckley's theorem (Part (b)).

\begin{theorem}\label{O:SharpB}

\emph{(a)} \emph{\cite[Theorem II.6(2)]{BeRe}} There exist real positive constants $c$ and $C$, independent of the weight $v$, 
such that for every weight $v$ and every 
interval $J$ we have
\begin{equation}
c\,m_J\left( v\log\left(\frac{v}{m_Jv}\right)\right) \leqslant \frac{1}{|J|}\sum_{I\in\mathcal{D}(J)}\frac{|\Delta_I v|^2}{m_Iv} |I|
\leqslant C\,m_J\left( v\log\left(\frac{v}{m_Jv}\right) \right)
\end{equation}

\noindent and as a consequence  $\|v\|_{L\log L, J} \approx \frac{1}{|J|}\sum_{I\in\mathcal{D}(J)}\frac{|\Delta_I v|^2}{m_Iv}|I|$.

\emph{(b)}    
  Let $v$ be
 a weight such that $v\in RH^d_1\,.$ Then $\{|\Delta_Iv|^2|I|/m_Iv\}_{I\in\mathcal{D}}$ is
a $v$-Carleson sequence with intensity comparable to $[v]_{RH^d_1}\,.$ That is, there is a constant $C>0$ such that for any
$J\in\mathcal{D}$,
$$\frac{1}{|J|}\sum_{I\in\mathcal{D}(J)}\frac{|\Delta_I v|^2}{m_Iv}|I|
\leq C[v]_{RH_1^d}m_Jv\,.$$
\end{theorem}


\section{Dyadic operators and known two weight results}

We now introduce several dyadic operators which will be considered in this paper, and record known two weight results for them.

\vskip .1in

\subsection{Dyadic weighted maximal function}
First we recall the dyadic weighted maximal function.

\begin{definition} We define the dyadic weighted maximal function $M^d_v$ as follows
$$M^d_vf(x):=\sup_{\substack{I\ni x\\I\in\mathcal{D}}}\frac{1}{v(I)}\int_I\,|f(y)|\,v(y)dy$$
\end{definition}

The weighted maximal function $M_v$ is defined analogously by taking the supremum
over all intervals not just dyadic intervals.
A very important fact about the weighted maximal function is that the $L^p(v)$ norm
of $M^d_v$ only depends on $p'=p/(p-1)$ not on the weight $v\,.$

\begin{theorem}\label{T:WMF}
Let $v$ be a locally integrable function such that $v>0$ a.e. Then for all $1<p<\infty$,
 $M^d_v$ is bounded in $L^p(v)$. Moreover, for all $f\in L^p(v)$
$$\|M^d_vg\|_{L^p(v)}\leq C_p\|f\|_{L^p(v)}\,.$$
\end{theorem}
This result follows by the Marcinkiewicz interpolation theorem, with constant $\displaystyle{C_p=2(p')^{\frac{1}{p}}}$, using the facts that $M^d_v$ is bounded on $L^{\infty}(v)$ with constant $1$ and it is weak-type $(1,1)$ also with constant $1$. Note that as $p \rightarrow 1$, $C_p \rightarrow 2p'$ and $C_2=2\sqrt{2}.$

When $v=1$, $M_1$ is the maximal function that we will denote $M$.
In \cite{Bu}, Buckley showed that the $L^p(w)$ norm of $M$ behaves like $[w]_{A_p}^{\frac{1}{p-1}}$, in particular the $L^2(w)$ norm of $M$ depends linearly
on the $A_2$ charateristic of the weight.  The next theorem is Sawyer's celebrated two weight result for the maximal function $M$ in the case $p=2$.

\begin{theorem}\cite{S1} The maximal function $M$ is bounded from $L^2(u)$ into $L^2(v)$ if and only if  
there is a constant $C_{u,v}>0$ such that
\begin{equation}\label{Stest}
\int_I \big (M( \mathbbm{1}_I u^{-1})(x)\big )^2v(x)\,dx \leq C_{u,v} u^{-1}(I), \quad\quad \mbox{for all intervals } \;\; I.\end{equation}
\end{theorem}
A quantitative version of this result was given by Moen, he showed in \cite{Moe} that the operator norm of $M$ from $L^2(u)$ into $L^2(v)$
is comparable to $2C_{u,v}$. 
Note that Sawyer's test condition (\ref{Stest}) implies  $(u,v)\in \mathcal{A}_2$, moreover $[u,v]_{\mathcal{A}_2}\leq C_{u,v}$.

A quantitative two weight result for the maximal function  not involving Sawyer's test conditions, instead involving  joint $\mathcal{A}_2$  and $RH_1$ constant of $u^{-1}$, has been recently found by P\'erez and Rela. 


\begin{theorem}\label{T:PzR}\cite{PzR} Let $u$ and $v$ be weights such that $(u,v)\in \mathcal{A}_2$ and $u^{-1}\in RH_1$ then
\[ \|M\|_{L^2(u)\to L^2(v)}\leq C ([u,v]_{\mathcal{A}_2}[u^{-1}]_{RH_1})^{1/2}.\]
\end{theorem}
This result is valid in certain spaces of homogeneous type, see \cite{PzR}. In fact they prove a result valid in $L^p$ replacing joint $A_2$ 
by joint $A_p$ and the power $1/2$ by $1/p$. More precisely they show 
  \[ \|M\|_{L^p(u)\to L^p(v)}\leq Cp' ([u,v]_{\mathcal{A}_p}[u^{-1}]_{RH_1})^{1/p},\]
where $p'=p/(p-1)$ is the dual exponent to $p$.

\subsection{Dyadic square function}

Second, we introduce the dyadic square function.

\begin{definition} We define the dyadic square function as follows
$$S^df(x):=\bigg(\sum_{I\in\mathcal{D}}|m_If-m_{\hat{I}}f|^2\mathbbm{1}_I(x)\bigg)^{1/2}\,,$$
where $\hat{I}$ denotes the dyadic parent of $I$.
\end{definition}

In \cite{HTV}, Hukovic, Treil and Volberg showed that the $L^2(v)$ norm of $S^d$ depends linearly on the $A_2$ characteristic of the weight.
Cruz-Uribe, Martell, and P\'erez  \cite{CrMPz2} showed that the $L^3(v)$ norm of $S^d$ depends on $[v]_{A_3}^{1/2}$. One concludes
that  $\|S^df\|_{L^p(v)}\leq C[v]_{A_p}^{\max\{\frac{1}{2},\frac{1}{p-1}\}}\|f\|_{L^p(v)}$ by sharp extrapolation \cite{DGPPet}, this bound is optimal.
 Lerner \cite{Le} has shown that  this holds for Wilson's intrinsic square function \cite{Wilbook}.

The following two weight characterization
was introduced by Wilson, see also \cite{NTV1}

\begin{theorem}\cite{Wil87}\label{T:TWSq} 
 The dyadic square function $S^d$ is
bounded from $L^2(u)$ into $L^2(v)$ if and only if
\begin{itemize}
\item[(i)] $(u,v)\in \mathcal{A}^d_2$
\item[(ii)] $\{|I|\,|\Delta_I u^{-1}|^2m_Iv\}_{I\in\mathcal{D}}$ is a $u^{-1}$-Carleson sequence with intensity $\mathcal{C}_{u,v}$ .
\end{itemize}
\end{theorem}

Condition (ii) can be viewed as a localized testing condition on the test functions $u^{-1}\mathbbm{1}_J$ for $J\in\mathcal{D}$.  Thus,
$\mathcal{C}_{u,v}\leq \|S^d\|^2_{L^2(u)\to L^2(v)}$.  


Recently Lacey and  Li \cite{LLi} showed a continuous quantitative analogue of  this theorem and they claim the dyadic version is ``a direct analog of their theorem",
their estimate  would read
\begin{equation}\label{LLiSquare}
\|S^d\|_{L^2(u)\to L^2(v)} \lesssim ([u,v]_{\mathcal{A}^d_2} + \mathcal{C}_{u,v})^{1/2}.
\end{equation}

We will present a proof of this estimate in Section~\ref{sec:square2}.
We will get quantitative two weight estimates for the dyadic square function involving  either the
two weight norm of the maximal operator and  $[v]^{1/2}_{RH^d_1}$, or
${[u,v]^{1/2}_{\mathcal{A}^d_2}}$,  $[u^{-1}]^{1/2}_{RH^d_1}$,  and $[v]^{1/2}_{RH^d_1}$, under appropriate assumptions in each case.

Theorem~\ref{O:SharpB}(b) implies that if $u^{-1}\in RH_1^d$ and $(u,v)\in \mathcal{A}_2^d$ then condition (ii) in Theorem~\ref{T:TWSq} holds
with $\mathcal{C}_{u,v}\lesssim [u,v]_{\mathcal{A}_2^d}[u^{-1}]_{RH_1^d}$. 
As a corollary of \eqref{LLiSquare} we get that if $u^{-1}\in RH_1^d$ and $(u,v)\in \mathcal{A}_2^d$ then

\begin{equation}\label{quantS}
 \|S^d\|_{L^2(u)\to L^2(v)} \leq C ([u,v]_{\mathcal{A}^d_2} +[u,v]_{\mathcal{A}^d_2}[u^{-1}]_{RH_1^d})^{1/2}.
\end{equation}

This improves  \cite[Theorem 4.1]{Be2} where the stronger assumption $u^{-1}\in A_q^d$ for some $q>1$ was made and a similar quantitative
two weight estimate was obtained with $[u^{-1}]_{A_q^d}$ replacing $[u^{-1}]_{RH_1^d}$ and the constant $C$ depending on $q$. Her results are proved in a setting where  the underlying Lebesgue measure is replaced by a doubling measure $\sigma$ on $\R$ (a space of homogeneous type),  introducing  
a dependence  on the doubling constant of $\sigma$ which is tracked in the aformentioned theorem.
We will prove \eqref{quantS}  without relying on \eqref{LLiSquare} in Section~\ref{sec:square2}. A closer look shows that the same argument
will allow us to recover \eqref{LLiSquare}.
  When $u=v=w\in A_2^d$ this improves  Hukovic's linear bound to a mixed bound:
\[ \|S^d\|_{L^2(w)}\leq C ([w]_{A_2^d}[w^{-1}]_{RH_1^d})^{1/2}.\]


\subsection{Martingale transform}

Third, we introduce the martingale transforms.

\begin{definition}
Let $r$ be a function from $\mathcal{D}$ into $\{-1,1\}$ so that $r(I)=r_I$, then we
define the martingale transform $T_r$ associated to $r$, acting on functions $f\in L^2(\R)$, by
$$T_rf(x):=\sum_{I\in\mathcal{D}}r_I\langle f,h_I\rangle h_I(x)\,.$$
\end{definition}

In \cite{W}, Wittwer showed that the $L^2(w)$ norm of $T_r$ depends linearly on
the $A_2$ characteristic of the weight $w$. The next theorem is from \cite{NTV1}
and it gives necessary and sufficient conditions for the martingale transforms $T_r$
to be uniformly bounded from $L^2(u)$ into $L^2(v)$. Before we state the theorem,
let us define the positive operator
$$T_0f(x):=\sum_{I\in\mathcal{D}}\frac{\alpha_I}{|I|}m_If\,\mathbbm{1}_I(x)\,,$$
where $\alpha_I=\frac{|\Delta_Iv|}{m_Iv}\frac{|\Delta_I(u^{-1})|}{m_I(u^{-1})}|I|\,.$

\begin{theorem}\label{NTVSq}\emph{\cite{NTV1}} The martingale transforms $T_r$ are
uniformly bounded from $L^2(u)$ to $L^2(v)$ if and only if the following four
assertions hold simultaneously:
\begin{itemize}
\item[(i)] $(u,v)\in \mathcal{A}_2$
\item[(ii)] $\{|I|\,|\Delta_I u^{-1}|^2m_Iv\}_{I\in\mathcal{D}}$ is a $u^{-1}$-Carleson sequence.
\item[(iii)] $\{|I|\,|\Delta_Iv|^2m_I(u^{-1})\}_{I\in\mathcal{D}}$ is a $v$-Carleson sequence.
\item[(iv)] The positive operator $T_0$ is bounded from $L^2(u)$ into $L^2(v)\,.$
\end{itemize}
\end{theorem}

As a corollary of the previous results in this section we can rewrite Theorem~\ref{NTVSq} as follows,

\begin{corollary}
The martingale transforms $T_r$ are uniformly
bounded from $L^2(u)$ to $L^2(v)$ if and only if the following three
assertions hold simultaneously:
\begin{itemize}
\item[(i)] $S^d$ is bounded from $L^2(u)$ into $L^2(v)$.
\item[(ii)] $S^d$ is bounded from $L^2(v^{-1})$ into $L^2(u^{-1})$.
\item[(iii)] The positive operator $T_0$ is bounded from $L^2(u)$ into $L^2(v)\,.$
\end{itemize}
\end{corollary}


\subsection{Dyadic paraproduct}
Finally we recall the definition of the dyadic paraproduct.

\begin{definition} We formally define the dyadic paraproduct  $\pi_b$  associated to $b\in L^1_{loc}(\R)$ as follows
for functions $f$  which are at least locally integrable:
$$\pi_bf(x):=\sum_{I\in\mathcal{D}} m_If \,\langle b,h_I\rangle h_I(x).$$
\end{definition}

It is a well know fact that the dyadic paraproduct is bounded not only on $L^p(dx)$ but
also on $L^p(v)$ when $b\in BMO^d$ and $v\in A_p^d$. Beznosova proved in \cite{Be1}
that the $L^2(v)$ norm of the dyadic paraproduct depends linearly
on both  $[v]_{A^d_2}$ and $\|b\|_{BMO^d}\,.$ Sharp extrapolation \cite{DGPPet} then shows 
\[ \|\pi_bf\|_{L^p(w)}\leq C\|b\|_{BMO^d} [w]_{A_p^d}^{\max\{1,\frac{1}{p-1}\}}\|f\|_{L^p(w)}.\]

When both weights $u,v\in A^d_p$ then it is known that the boundedness of the dyadic paraproduct $\pi_b:L^p(u)\to L^p(v)$ is equivalent
to $b$ being in  a weighted $BMO^d(\mu)$ where $\mu=u^{1/p}v^{-1/p}$, that is,
\begin{equation}\label{BloomBMO}
\|b\|_{BMO^d(\mu)}:= \sup_{I\in\mathcal{D}}\frac{1}{\mu(I)}\int_I |b(x) - m_Ib|\,dx <\infty.
\end{equation}
This space is known as Bloom's $BMO$ \cite{Bl}.
In fact there are a number of conditions equivalent to \eqref{BloomBMO} (see \cite{HoLWic1}) one of them being
the boundedness of the adjoint of the dyadic paraproduct $\pi_b^*:L^p(u)\to L^p(v)$. By duality the last result
is equivalent to the boundedness of the dyadic paraproduct $\pi_b:L^{p'}(v')\to L^{p'}(u')$, where $p,p'$ are dual exponents, $\frac{1}{p}+\frac{1}{p'}=1$,
and $u', v'$ are dual weights, namely
 $u'=u^{\frac{-1}{p-1}}=u^{-p'/p}$. Not surprisingly $\mu'= (v')^{1/p'}(u')^{-1/p'} = \mu$, so that $BMO(\mu')=BMO(\mu)$. 
The assumption that both weights are in
$A_p^d$ is very symmetric and forces boundedness of the paraproduct and its adjoint to occur simultaneously. This is 
the appropriate setting when dealing with two-weight inequalities for commutators which very naturally can be separated into commutators with 
a paraproduct, its adjoint, and other terms which  will  all be bounded from $L^2(u)$ into $L^2(v)$ provided
$u,v\in A_p^d$ and $b$ is in Bloom's $BMO(\mu)$. Assuming both $u,v\in A_p^d$ allows one to use Littlewood-Paley theory for the dyadic square function $S^d$, specifically,
the $L^p(w)$ norm of $S^dg$ is comparable to the $L^p(w)$ norm of $g$ whenever $w\in A_p^d$. 
In particular  $\|\pi_bf\|^2_{L^2(v)}$ is comparable to $\|S^d(\pi_bf)\|^2_{L^2(v)}= \sum_{I\in\mathcal{D}}|m_If|^2b_I^2m_I(v)$, and from here
boundedness from $L^2(u)$ into $L^2(v)$ of the dyadic paraproduct is reduced to verifying the following estimate
\[ \sum_{I\in\mathcal{D}}|m_If|^2b_I^2m_I(v) \leq C_{u,v,b} \|f\|^2_{L^2(u)}.\]
This inequality holds by the weighted Carleson lemma (Lemma~\ref{WCL}) and the boundedness of the maximal function in $L^2(u)$ when $u\in A_2^d$, provided
the sequence $\{b_I^2m_I(v)\}_{I\in\mathcal{D}}$ is a $u$-Carleson sequence, namely
\begin{equation}\label{condition1vA2} 
\sum_{I\in\mathcal{D}(J)}b_I^2m_I(v) \leq C u(J).
\end{equation}
Another use of the Littlewood-Paley theory ($v\in A_2$) allows us to compare the left-hand-side to
$\int_J |b(x)-m_Jb|^2v(x)\,dx$ yielding what turns out is an equivalent condition for the boundedness of the paraproduct from $L^2(u)\to L^2(v)$ when  $u, v\in A^d_2$ 
(see \cite{HoLWic1})
\begin{equation}\label{condition2vA2}
 \sup_{J\in\mathcal{D}}\frac{1}{u(J)}\int_J |b(x)-m_Ib|^2v(x)\,dx <\infty.
\end{equation}
In \cite[Theorem 3.1]{HoLWic2} the authors present an equivalent condition for the boundedness of the paraproduct from $L^2(u)\to L^2(v)$
when only $v\in A_2^d$, namely 
\begin{equation}\label{condition3vA2}
\mathcal{B}_2(u,v):=\sup_{J\in\mathcal{D}}\frac{1}{u^{-1}(J)}\sum_{I\in\mathcal{D}(J)}b_I^2(m_Iu^{-1})^2m_I(v)<\infty.
\end{equation}
Conditions~\eqref{condition2vA2} and \eqref{condition3vA2}  are testing conditions  for the test functions $u^{-1}\mathbbm{1}_J$.


In Section \ref{Paraproduct} we provide sufficient conditions
on a pair of weights $(u,v)$ for the two weight boundedness of the dyadic paraproduct operator from $L^2(u)$ into $L^2(v)$
when $b\in Carl_{u,v}$, together with a quantitative estimate. The conditions we consider are less symmetric, we assume a priori
that  $(u,v)\in \mathcal{A}_2^d$ (which is equivalent to $(v^{-1},u^{-1})\in \mathcal{A}_2^d$), and an assymetric weighted Carleson condition, or equivalently we assume
the dyadic square $S^d$ function is bounded from $L^2(v^{-1})\to L^2(u^{-1})$.
Under these conditions we show that if $b\in Carl_{u,v}$ then $\pi_b$ is bounded from $L^2(u)$ into $L^2(v)$. We would have liked to show that $b\in Carl_{u,v}$
is not only a sufficient condition but also a necessary condition for the boundedness of the dyadic paraproduct under the a priori assumptions on the 
pair of weights, but we have not been able to identify the appropriate testing functions that will yield this result.
If we wish to show that both the paraproduct and its adjoint are bounded from $L^2(u)$ into $L^2(v)$ then we need to assume a priori joint $\mathcal{A}_2$ 
and two mixed 
 Carleson conditions on the weights, and we need to assume  $b\in Carl_{u,v}\cap Carl_{v^{-1},u^{-1}}$. It will be interesting to compare
these conditions, for example can one show that if $u,v\in A_2^d$ then Bloom's $BMO$ coincides with $b\in Carl_{u,v}\cap Carl_{v^{-1},u^{-1}}$? Can we conclude that when $(u,v)\in \mathcal{A}_2^d$  then $Carl_{u,v}$ is equivalent to $\mathcal{B}_2(u,v) <\infty$? or that when $v\in A_2^d$
then $Carl_{v,v}$ is equivalent to $\mathcal{B}_2(v,v) <\infty$?
We record some results comparing these conditions in Section~\ref{sec:CarlvsBloom}.


%


\section{The dyadic paraproduct, bump conditions, and  $BMO$ vs $Carl_{u,v}$}\label{Paraproduct}

In this section we will state and prove our main two weight result about the dyadic paraproduct (Theorem 1.1 in the introduction, 
called Theorem~\ref{T:para2} in this section ). We will also compare our result to known two weight bump conditions,
compare the class $Carl_{v,v}$ with $BMO^d$ when $v\in A_2^d$, and compare the class $Carl_{u,v}\cap Carl_{v^{-1},u^{-1}}$ with 
Bloom's $BMO$ when both $u,v\in A_2^d$.

\subsection{Two weight estimate for the dyadic paraproduct}
In this section we obtain quantitative two-weight estimates for the dyadic paraproduct
$\pi_b$ when
$b\in Carl_{u,v}$ and $(u,v)$ are two weights with some additional conditions.
Note that by definition $b$ is 
a locally integrable function, thus $b_I=\langle b,h_I\rangle$ is well defined.


\begin{theorem}\label{T:para2} Let $(u,v)$ be a pair of functions such that $v$ and
$u^{-1}$ are  weights, 
  $(u,v)\in \mathcal{A}_2^d$, and
$\{|\Delta_I v|^2|I| m_I(u^{-1})\}_{I\in\mathcal{D}}$ is a $v$-Carleson
sequence with intensity $\mathcal{D}_{u,v}$. Then $\pi_b$ is bounded from
$L^2(u)$ into $L^2(v)$ if $b\in Carl_{u,v}$.
Moreover, if $\mathcal{B}_{u,v}$ is the intensity of the $u^{-1}$-Carleson
sequence $\{|b_I|^2/m_Iv\}_{I\in\mathcal{D}}$ then there exists $C>0$ such that for all $f\in L^2(u)$
$$\|\pi_bf\|_{L^2(v)}\leq C\sqrt{[u,v]_{\mathcal{A}_2^d}\mathcal{B}_{u,v}}
\Big(\sqrt{[u,v]_{\mathcal{A}_2^d}}+\sqrt{\mathcal{D}_{u,v}}\,\Big)\|f\|_{L^2(u)}\,.$$
\end{theorem}


\begin{proof}
Fix $f\in L^2(u^{-1})$ and $g\in L^2(v)\,,$ then $fu^{-1}\in L^2(u)$, 
$\|fu^{-1}\|_{L^2(u)}=\|f\|_{L^2(u^{-1})}$, $gv\in L^2(v^{-1})$ and
$\|gv\|_{L^2(v^{-1})}=\|g\|_{L^2(v)}\,,$ $\pi_b(fu^{-1})$ is expected to be
in $L^2(v)\,,$ then $gv\in L^2(v^{-1})$ is in the right space for the pairing.
Thus, by duality, suffices to prove:
\begin{equation}|\langle \pi_b(fu^{-1}),gv\rangle|
\leq C\sqrt{[u,v]_{\mathcal{A}_2^d}\mathcal{B}_{u,v}}
\Big(\sqrt{[u,v]_{\mathcal{A}_2^d}}+\sqrt{\mathcal{D}_{u,v}}\,\Big)
\|f\|_{L^2(u^{-1})}\|g\|_{L^2(v)}\,.\label{DualofP}
\end{equation}

\noindent Replace $h_I$ by $\alpha_I h_I^v+\beta_I\frac{\mathbbm{1}_I}{\sqrt{|I|}}$
where $\alpha_I=\alpha^v_I$ and $\beta_I=\beta_I^v$ as described in
Proposition \ref{WHB}, to get
\begin{equation}
|\langle \pi_b(fu^{-1}),gv)\rangle|\leq \sum_{I\in\mathcal{D}}|b_I|m_I(|f|u^{-1})
\bigg|\bigg\langle gv,\alpha_I h^v_I+\beta_I\frac{\mathbbm{1}_I}{\sqrt{|I|}}\bigg\rangle\bigg|.\label{DualofPS}
\end{equation}

\noindent Use the triangle inequality to separate the sum in \eqref{DualofPS}
into two summands 
$$|\langle \pi_b(fu^{-1}),gv\rangle|\leq \sum_{I\in\mathcal{D}}
|b_I|
|\alpha_I|m_I(|f|u^{-1})|\langle gv,h^v_I\rangle|
+\sum_{I\in\mathcal{D}}|b_I|\frac{|\beta_I|}{\sqrt{|I|}}m_I(|f|u^{-1})|\langle gv,\mathbbm{1}_I\rangle|\,.$$

\noindent Using the estimates $|\alpha_I|\leq \sqrt{m_Iv}\,$ and
$|\beta_I|\leq \frac{|\Delta_Iv|}{m_Iv}\,$  in
Proposition \ref{WHB}, we have that,
$$|\langle\pi_b(fu^{-1}),gv\rangle|\leq \Sigma_1+\Sigma_2,$$
where
\noindent \begin{align*}
&\Sigma_1:=\sum_{I\in\mathcal{D}}|b_I|m_I(|f|u^{-1})|\langle gv,h^v_I\rangle|\sqrt{m_Iv}\\
&\Sigma_2:=\sum_{I\in\mathcal{D}}|b_I|m_I(|f|u^{-1})|\langle gv,\mathbbm{1}_I\rangle|
\frac{|\Delta_Iv|}{m_Iv}\frac{1}{\sqrt{|I|}}\,.
\end{align*}

\noindent\textbf{Estimating} $\Sigma_1$: We have
\begin{align*}
\Sigma_1&\leq \sum_{I\in\mathcal{D}}\frac{|b_I|}{\sqrt{m_Iv}}m^{u^{-1}}_I(|f|)|
\langle g,h^v_I\rangle_v|m_I(u^{-1})m_Iv\\
&\leq [u,v]_{\mathcal{A}_2^d}\sum_{I\in\mathcal{D}}\frac{b_I}{\sqrt{m_Iv}}
\inf_{x\in I}M_{u^{-1}}f(x)|\langle g ,h^v_I\rangle_v|\\
&\leq [u,v]_{\mathcal{A}_2^d}\bigg(\sum_{I\in\mathcal{D}}\frac{|b_I|^2}{m_Iv}
\inf_{x\in I}M^2_{u^{-1}}f(x)\bigg)^{1/2}
\bigg(\sum_{I\in\mathcal{D}}|\langle g,h^v\rangle_v|^2\bigg)^{1/2}\,.
\end{align*}
Here in the first line we use that $\langle gv,f\rangle=\langle g,f\rangle_v$, in the second line we use that
$m_I^{u^{-1}}|f|:=\frac{m_I(|f|u^{-1})}{m_I(u^{-1})}\leq M_{u^{-1}}f(x)$
for all $x\in I\,,$ and  that $m_I(u^{-1})m_Iv\leq [u,v]_{\mathcal{A}_2^d}$
and in the third line we use the Cauchy-Schwarz inequality.

\noindent Using the fact that $\{h_I^v\}_{I\in\mathcal{D}}$ is an orthonormal system in $L^2(v)$ and the Weighted Carleson Lemma~\ref{WCL}, with
$F(x)=M^2_{u^{-1}}f(x)\,,$ and $\alpha_I=|b_I|^2/m_Iv$, which is a $u^{-1}$-Carleson
sequence with intensity $\mathcal{B}_{u,v}$, by  assumption, we get
\begin{align}
\Sigma_1&\leq [u,v]_{\mathcal{A}^d_2}\sqrt{\mathcal{B}_{u,v}}
\bigg(\int_{\R}M^2_{u^{-1}}f(x)u^{-1}(x)dx\bigg)^{1/2}\|g\|_{L^2(v)} \nonumber \\
&\leq 2\sqrt{2}[u,v]_{\mathcal{A}^d_2}\sqrt{\mathcal{B}_{u,v}}\|f\|_{L^2(u^{-1})}\|g\|_{L^2(v)}\,. \label{sigma_1}
\end{align}

\noindent In the second inequality we used Theorem \ref{T:WMF}.
\vskip .1in

\noindent\textbf{Estimating} $\Sigma_2$: Using similar arguments as the
ones used for $\Sigma_1\,,$ we conclude that,
\begin{align*}
\Sigma_2&\leq \sum_{I\in\mathcal{D}}|b_I|m^{u^{-1}}_I(|f|)m_I^v(|g|)
\frac{|\Delta_Iv|}{m_Iv}\sqrt{|I|}m_I(u^{-1})m_Iv\\
&=\sum_{I\in\mathcal{D}}\frac{|b_I|}{\sqrt{m_Iv}}m^{u^{-1}}_I(|f|)m_I^v(|g|)
|\Delta_Iv|\sqrt{|I|}m_I(u^{-1})\sqrt{m_Iv}\\
&\leq [u,v]_{\mathcal{A}_2^d}^{1/2}\sum_{I\in\mathcal{D}}\frac{|b_I|}{\sqrt{m_Iv}}
|\Delta_Iv|\sqrt{|I|}\sqrt{m_Iu^{-1}}\inf_{x\in I}M_{u^{-1}}f(x)\inf_{x\in I}M_vg(x)\\
&\leq [u,v]_{\mathcal{A}_2^d}^{1/2}\bigg(\sum_{I\in\mathcal{D}}\frac{|b_I|^2}{m_Iv}
\inf_{x\in I}M^2_{u^{-1}}f(x)\bigg)^{1/2}\bigg(\sum_{I\in\mathcal{D}}
|\Delta_Iv|^2m_I(u^{-1})|I|\inf_{x\in I}M^2_vg(x)\bigg)^{1/2}\,.
\end{align*}

\noindent By hypothesis $\{|b_I|^2/m_Iv\}_{I\in\mathcal{D}}$ is a
$u^{-1}$-Carleson sequence and $\{|\Delta_Iv|\,|I|m_I(u^{-1})\}_{I\in\mathcal{D}}$
is a $v$-Carleson sequence with intensities $\mathcal{B}_{u,v}$ and $\mathcal{D}_{u,v}$
respectively. By Lemma (\ref{WCL}),
\begin{align*}
\Sigma_2&\leq\sqrt{[u,v]_{\mathcal{A}_2^d}\mathcal{B}_{u,v}\mathcal{D}_{u,v}}
\bigg(\int_{\R}M^2_{u^{-1}}f(x)u^{-1}(x)dx\bigg)^{1/2}\bigg(\int_{\R}M_v^2g(x)v(x)dx\bigg)^{1/2}\\
&\leq\sqrt{[u,v]_{\mathcal{A}_2^d}\mathcal{B}_{u,v}\mathcal{D}_{u,v}}\|M_{u^{-1}}f\|_{L^2(u^{-1})}\|M_vg\|_{L^2(v)}\\
&\leq 8\sqrt{[u,v]_{\mathcal{A}_2^d}\mathcal{B}_{u,v}\mathcal{D}_{u,v}}\|f\|_{L^2(u^{-1})}\|g\|_{L^2(v)}\,.
\end{align*}
This estimate, together with estimate (\ref{sigma_1}), gives (\ref{DualofP}).
\end{proof}
%
%

%
We can replace the conditions on the pair $(u,v)$ by boundedness of the dyadic square function to deduce boundedness of the dyadic paraproduct
when $b\in Carl_{u,v}$.

\begin{corollary} Let $b\in L^1_{loc}(\R )$ and $(u,v)$ be a pair of functions such that $v$ and $u^{-1}$ are weights  and
 $\{|b_I|^2/m_Iv\}_{I\in\mathcal{D}}$ is a $u^{-1}$-Carleson sequence $(b\in Carl_{u,v})$
 with  intensity $\mathcal{B}_{u,v}$. If the dyadic square function $S^d$ is bounded from
$L^2(v^{-1})$ into $L^2(u^{-1})$ then the paraproduct $\pi_b$ is bounded from $L^2(u)$ into $L^2(v)\,.$
Moreover
\[ \|\pi_bf\|_{L^2(v)}\leq C\sqrt{[u,v]_{\mathcal{A}_2^d}\mathcal{B}_{u,v}}
\Big(\sqrt{[u,v]_{\mathcal{A}_2^d}}+{\|S^d\|_{L^2(v^{-1})\to L^2(u^{-1})}}\,\Big)\|f\|_{L^2(u)}\,.\]
\end{corollary}

\begin{proof}
Assume $S^d$ is bounded from $L^2(v^{-1})$ into $L^2(u^{-1})$. Theorem \ref{NTVSq}
implies that $(u,v)\in \mathcal{A}_2$ and $\{|\Delta_Iv|^2|I|m_I(u^{-1})\}_{I\in\mathcal{D}}$ is
$v$-Carleson sequence  with intensity $\mathcal{C}_{v^{-1},u^{-1}}$. Moreover, $\mathcal{C}_{v^{-1},u^{-1}}\leq \|S^d\|^2_{L^2(v^{-1})\to L^2(u^{-1})}$.
 These two facts together with the hypothesis that
$\{|b_I|^2/m_Iv\}_{I\in\mathcal{D}}$ is a $u^{-1}$-Carleson sequence imply, by
Theorem \ref{T:para1}, that $\pi_b$ is bounded from $L^2(u)$ to $L^2(v)\,.$ The claimed estimate holds.
\end{proof}

If we especialize to the one weight case $u=v=w\in A_2^d$ then $\|S^d\|_{L^2(w^{-1})}\leq C [w^{-1}]_{A_2^d}=C [w]_{A_2^d}$. Moreover, $b \in Carl_{w,w}\cap L^2_{loc}$ is equivalent to $b \in BMO^d$ and $\mathcal{B}_{w,w} \leq C\|b\|^2_{BMO^d}$, we show this in Corollary~\ref{C:BMOCvv}. The previous Corollary would give us that 
\[ \|\pi_b\|_{L^2(w) \rightarrow L^2(w)} \leq C \|b\|_{BMO^d} [w]_{A_2^d}^{\frac32}. \] 
Thus, we do not recover Beznosova's linear bound, we are off by $\displaystyle {[w]_{A_2^d}^{\frac12}}$.


\subsection{Comparison to one-sided bump theorems}

The dyadic paraproduct is especially interesting because it allows us to estimate Calder\'on-Zygmund singular integral operators (CZSIO).
The general approach to the two weight estimates for the CZSIO as a class is a bump-approach. We refer the reader to \cite{NRV}
 for the precise definitions and statements, the interested reader can also consult \cite{V} in this volume.

\begin{theorem}\cite[Theorem 3.2]{NRV}\label{theorem:NVR}
Suppose $\Phi$ satisfies several conditions\footnote{The conditions on the function $\Phi$ are satisfied by the functions $\Phi (L) = L \log^{1+\sigma} L$ and $L \log L \log \log^{1+\sigma}L$ (for sufficiently large $\sigma>0$), but not by $\Phi(L) = L\log L$.}. Suppose that there exists a constant $C$ such that for all
$I\in\mathcal{D}$
\begin{equation}
\|u^{-1}\|_{L,I} \|v\|_{\Phi(L),I} \leqslant C.
\end{equation}
Then any  Calder\'on-Zygmund singular integral operator $T$ is weakly bounded from $L^2(u)$ into $L^{2,\infty} (v)$, i.e.,
\begin{equation}
v\{x\in\R: |Tf(x)|\geq \lambda\}\leqslant \left (\frac{C \|f\|_{L^2(u)}}{\lambda}\right )^2.
\end{equation}
\end{theorem}

Let us assume that $u$ and $v$ are such that
$$
\|u^{-1}\|_{L,J} \|v\|_{L\log L,J} \leqslant C,
$$
which is a weaker condition than the condition in Theorem \ref{theorem:NVR}.
Then by  Theorem \ref{O:SharpB} 
we have that, for every $J\in\mathcal{D} $,
\begin{equation}\label{e:bump}
\|v\|_{L\log L, J} \approx \frac{1}{|J|}\sum_{I\in\mathcal{D}(J)}\frac{|\Delta_I v|^2}{m_Iv}|I| \leqslant \frac{C}{m_J(u^{-1})} .
\end{equation}
The condition we have for the paraproduct is
\begin{equation}\label{e:our}
\frac{1}{|J|} \sum_{I \in D(J)} |\Delta_I v|^2 m_I(u^{-1}) |I| \leqslant C m_J(v)
\end{equation}
Note that if $(u,v)\in \mathcal{A}_2^d$ 
we have that
$$
\frac{1}{|J|} \sum_{I \in D(J)} |\Delta_I v|^2 m_I(u^{-1}) |I| \leqslant {[u,v]_{\mathcal{A}_2^d}} \frac{1}{|J|}\sum_{I\in\mathcal{D}(J)}\frac{|\Delta_I v|^2}{m_Iv}|I|
$$
\noindent while $m_Jv \leqslant \frac{[u,v]_{A_2^d}}{m_J(u^{-1})}$. Therefore we cannot compare bump conditions to the conditions 
in our results without the additional assumption that there is a constant $q>0$ such that  $m_J(u^{-1})m_Jv \geqslant q $ for all $J\in\mathcal{D}$.
 If $q \leqslant m_J(u^{-1})m_Jv \leqslant Q$ for all $J\in\mathcal{D}$;
the two conditions  (\ref{e:bump}) and (\ref{e:our}) become equivalent, but this assumption essentially reduces the problem to the one weight  case
\cite[Proposition 7.4]{M}.

\subsection{$BMO$ vs $\mathcal{C}_{v,v}$}

Formally the dyadic paraproduct is a bilinear operator for the locally integrable
functions $b$ and $f\,.$ After we fix $b$ in $BMO^d$, we consider $\pi_b$ as a linear
operator acting on $f\,.$ In the following proposition, we try to answer the 
question: \emph{if $\pi_b$ is bounded on (weighted) Lebesgue spaces, then in which
space does the locally square integrable function $b$ lie?}

\begin{proposition}[A necessary condition for boundedness of $\pi_b$]
Let $u$ and $v$ be weights and, for $1<p<\infty$, $b\in L^2_{loc}(\R )$.  Assume
$\pi_b:L^p(u)\rightarrow L^p(v)$ is a bounded operator then  there is a constant $C_p>0$ such that for any $I\in\mathcal{D}\,,$
\begin{equation}
\int_I|b(x)-m_Ib|^p v(x)dx\leq C_pu(\hat{I})\,,\label{wbmo}
\end{equation}
\noindent where $\hat{I}$ is the dyadic parent of $I$. The constant $C_p^{1/p}$ is the operator norm  $\|\pi_b\|_{L^p(u)\to L^p(v)}$.
\end{proposition}

\begin{proof}
Let us choose $f=h_J$ for some dyadic interval $J$. Then, by assumption,
there exists a constant $C_p=\|\pi_b\|^p_{L^p(u)\to L^p(v)}$ such that
\begin{equation}
\int_{\R}|\pi_b(h_J)(x)|^pv(x)dx\leq C_p\int_{\R}|h_J(x)|^pu(x)dx=C_p\frac{u(J)}{|J|^{p/2}}\,.
\end{equation}
On the other hand,
\begin{eqnarray*}
\pi_b(h_J)(x)&=&\sum_{I\in\mathcal{D}}m_I(h_J)\langle b,h_I\rangle h_I(x)\\
&=&\sum_{I\in\mathcal{D}(J_+)}\frac{1}{\sqrt{|J|}}\langle b,h_I\rangle h_I(x)-
\sum_{I\in\mathcal{D}(J_-)}\frac{1}{\sqrt{|J|}}\langle b,h_I\rangle h_I(x)\\
&=&\frac{1}{\sqrt{|J|}}\Big[(b(x)-m_{J_+}b)\mathbbm{1}_{J_+}(x)-(b(x)-m_{J_-}b)\mathbbm{1}_{J_-}(x)\Big]\,,
\end{eqnarray*}
\noindent where the last equality is due to the fact that $(b - m_J b)\mathbbm{1}_J = \sum_{I\in D(J)} \langle b, h_I \rangle h_I$. Therefore we can write
\begin{eqnarray*}
\int_{\R}|\pi_b(h_J)(x)|^pu(x)dx
&=&\frac{1}{|J|^{p/2}}\int_{\R}\big|(b(x)-m_{J_+}b)\mathbbm{1}_{J_+}(x)-(b(x)-m_{J_-}b)\mathbbm{1}_{J_-}(x)\big|^pv(x)dx\\
&=&\frac{1}{|J|^{p/2}}\left(\int_{J_+}\big|b(x)-m_{J_+}b|^pv(x)dx+\int_{J_-}|b(x)-m_{J_-}b|^pv(x)dx\right)\,.
\end{eqnarray*}

\noindent Thus we can conclude that there is a constant $C_p$ such that for all $I\in\mathcal{D}$
$$\int_I|b(x)-m_Ib|^pv(x)dx\leq C_pu(\hat{I})\,.$$
\end{proof}

The condition \eqref{wbmo} can be considered as a testing condition for the boundedness of the dyadic paraproduct
from $L^p(u)$ into $L^p(v)$. When $u,v\in A_p^d$ both weights are doubling weights, in particular $u(\hat{I})\leq D(u)\,u(I)$
(where $D(u):=\sup_{I\in\mathcal{D}}u(\hat{I})/u(I) <\infty$ is the dyadic doubling constant of $u$).
In this case,  \eqref{wbmo} becomes 
\[ \int_I|b(x)-m_Ib|^pv(x)\,dx \leq C_p u(I)\]
 which is equivalent to the  boundedness of the paraproduct and its adjoint (\cite[Theorem 4.1]{HoLWic1})  from $L^p(u)$ into $L^p(v)$ when $u,v\in A_p^d$. When $u=v$ this necessary condition was known in the more general matrix $A_p$ context \cite{IKP}.

One can immediately conclude that the inequality (\ref{wbmo}) implies that $b$
is in $BMO^d$ for $u=v=1$ (Lebesgue space). Thus, one can view the condition $b\in BMO^d$ as a testing condition
for the boundedness of the paraproduct on $L^2(\R )$, in the same way that the conditions $T1, T^*1\in BMO$
in the celebrated $T1$ Theorem are testing conditions.

For the weighted Lebesgue space, we
have the following corollary.
\begin{corollary}\label{C:BMOd}
For $1<p<\infty\,,$  $b\in L^2_{loc}(\R)$, if $\pi_b$ is bounded from $L^p(v)$ into itself and $v$ is
an $A_p^d$ weight, then $b$ belongs to $BMO^d$. Moreover, $\|b\|_{BMO^d}\leq 2 \|\pi_b\|_{L^p(v)\to L^p(v)}[v]_{A_p^d}^{1/p}$.
\end{corollary}
\begin{proof}
For any $I\in\mathcal{D},$ we have
\begin{align}
\int_I|b(x)-m_Ib|\,dx&=\int_I|b(x)-m_Ib|v^{\frac{1}{p}}(x)v^{-\frac{1}{p}}(x)\,dx\nonumber\\
&\leq \bigg(\int_I|b(x)-m_Ib|^pv(x)\,dx\bigg)^{\frac{1}{p}}
    \bigg(\int_I v^{-\frac{p'}{p}}(x)\,dx\bigg)^{\frac{1}{p'}}\nonumber\\
&\leq C_p^{1/p} \bigg(\int_{\hat{I}}v(x)\,dx\bigg)^{\frac{1}{p}}
   \bigg(\int_{\hat{I}}v^{-\frac{p'}{p}}(x)\,dx\bigg)^{\frac{1}{p'}}\label{E:bmo}\\
&=C_p^{\frac{1}{p}} |\hat{I}|\bigg(\frac{1}{|\hat{I}|}\int_{\hat{I}}v(x)\,dx\bigg)^{\frac{1}{p}}
   \bigg(\frac{1}{|\hat{I}|}\int_{\hat{I}}v^{-\frac{1}{p-1}}(x)\,dx\bigg)^{\frac{p-1}{p}}\,\label{JointIn}\\
&\leq 2\|\pi_b\|_{L^p(v)\to L^p(v)} [v]_{A_p^d}^{\frac{1}{p}}|I|\,.\nonumber
\end{align}

\noindent Here the inequality (\ref{E:bmo}) holds due to (\ref{wbmo}) with $v=u\,.$
\end{proof}
Notice that $b\in BMO$ implies that $b\in L^p_{loc}(\R)$  for all $1\leq p<\infty$ by the John-Nirenberg inequality.

For the two weight case, in order to show that (\ref{JointIn}) is bounded, we need
$(v,u)\in \mathcal{A}_p$ which is totally different from $(u,v)\in \mathcal{A}_p$. Thus, we cannot
conclude anything more than (\ref{wbmo}) for the two weight situation. 

To finish  this section, we  give a relation between $BMO^d$ and $Carl_{v,v}$.

\begin{corollary} \label{C:BMOCvv}
If $v\in A^d_2$ then
$$BMO^d=Carl_{v,v}\cap L^2_{loc}(\R).$$
\end{corollary}
\begin{proof}
In Section \ref{S:Def_Carl}, we observed that $BMO^d\subset Carl_{v,v}$
for any weight $v$ such that $v^{-1}$ is also a weight.  Also recall that by the John-Nirenberg theorem if $b\in BMO$ then $b\in  L^2_{loc}(\R)$. Thus, to complete the proof,
we need to show that if $v\in A_2^d$ and $b\in Carl_{v,v}\cap  L^2_{loc}(\R)$ then $b\in BMO^d\,.$ If
$v\in A_2^d\,$ then in particular $v \in RH_1^d$. By Theorem \ref{O:SharpB}, it follows that, for every dyadic interval $J$, we have
\begin{equation}\label{eq:Cvv}
\frac{1}{|J|} \sum_{I\in D(J)} {|\Delta_I v|^2}{m_I(v^{-1})}|I| \leqslant [v]_{A_2^d}\frac{1}{|J|} \sum_{I\in D(J)} \frac{|\Delta_I v|^2}{m_Iv}|I|\leqslant C [v]_{A_2^d} [v]_{RH_1^d} m_Jv\;.
\end{equation}
Since $v \in A_2^d$ and $b \in Carl_{v,v}$, all conditions of  Theorem \ref{T:para1} are satisfied, and we know that the dyadic paraproduct,
$\pi_b$, is bounded from $L^2(v)$ into $L^2(v)$. Thus, by Corollary \ref{C:BMOd},
$b$ must belong to $BMO^d\,.$
\end{proof}
Similar one weight results are shown  by Isralowitz, Kwon, and Pott \cite{IKP} in the much more general matrix $A_p$ context.

\subsection{$Carl_{u,v}$ vs Bloom's $BMO$}\label{sec:CarlvsBloom}
There are other weighted bounded mean oscillation spaces in the literature. The weighted $BMO$ space for a weight $\mu$ in $\R^d$, denoted $BMO^d(\mu)$ in \cite[Section 2.6]{HoLWic1},
  consists of all locally integrable functions $b$ such that
\[\|b\|_{BMO^d(\mu)}:=\sup_{Q}\frac{1}{\mu(Q)}\int_Q|b(x)-m_Qb|\, dx <\infty,\]
where the supremum is taken over all cubes with sides parallel to the axes. In that paper, it is pointed out that when 
the weight is in $A_{\infty}$ (hence, in particular, is a doubling weight), one can replace the $L^1$ with
$L^p$ norm provided the integration with respect to the Lebesgue measure is replaced by $\sigma \,dx$ where $\sigma=\mu^{\frac{-1}{p-1}}$ is the conjugate weight. 

When $u,v\in A^d_2$, let $\mu:=u^{1/2}v^{-1/2}$, the corresponding weighted $BMO^d(\mu )$ is Bloom's $BMO$ \cite{Bl}. In \cite[Theorem 4.1]{HoLWic1} it is shown that the following 
  are equivalent conditions.
\begin{itemize}
\item[(i)] $b\in BMO^d(\mu )$.
\item[(ii)] $b\in BMO^d_2(\mu )$ meaning $\displaystyle{\sup_{I\in\mathcal{D}} \frac{1}{\mu (I)}\int_I|b(x)-m_Ib|^2\mu^{-1}(x)\,dx <\infty}$.
\item[(iii)] $\displaystyle{\sup_{I\in\mathcal{D}} \frac{1}{u(I)}\int_I|b(x)-m_Ib|^2v(x)\,dx <\infty}$.
\item[(iv)] $\displaystyle{\sup_{I\in\mathcal{D}} \frac{1}{v^{-1}(I)}\int_I|b(x)-m_Ib|^2u^{-1}(x)\,dx <\infty}$.
\item[(v)]$\pi_b$ is bounded from $L^2(u)$ into $L^2(v)$.
\item[(vi)] $\pi^*_b$ from $L^2(u)$ into $L^2(v)$.
\end{itemize}

\begin{theorem} Assume $u,v\in A_2^d$ and let $\mu=u^{1/2}v^{-1/2}$.
Then  $b\in Carl_{u,v}$ if and only if $b\in Carl_{v^{-1},u^{-1}}$  
 if and only if  $b\in BMO^d(\mu^{-1})$.

\end{theorem}
\begin{proof} First we will show that $Carl_{u,v}\cup Carl_{v^{-1},u^{-1}} \subset BMO^d(\mu^{-1})$.
Assume $b\in Carl_{u,v}\cup Carl_{v^{-1},u^{-1}}$. 
By assumption there is $C>0$ such that for all $J\in\mathcal{D}$
\[ (a)\; \sum_{I\in\mathcal{D}(J)} \frac{b_I^2}{m_Iv}\leq C u^{-1}(J), \quad \mbox{or} \quad (b)\;\sum_{I\in\mathcal{D}(J)} \frac{b_I^2}{m_Iu^{-1}}\leq C v(J).\]

When $w\in A_2^d$ the dyadic square function $S^d$ obeys an inverse estimate
$ \|f\|_{L^2(w)}\leq C [w]_{A_2^d}^{1/2}\|S^df\|_{L^2(w)}$.
 In case (a),  since $v^{-1}\in A_2^d$ we can use the inverse estimate for $S^d$ in $L^2(v^{-1})$ and get,  for all $J\in\mathcal{D}$, the estimate
\begin{eqnarray*}
 \|(b-m_Jb)\mathbbm{1}_J\|^2_{L^2(v^{-1})}&\leq & C[v]_{A_2^d} \|S^d\big ( (b-m_Jb)\mathbbm{1}_J\big )\|^2_{L^2(v^{-1})}\\
                   &= & C[v]_{A_2^d}\sum_{I\in\mathcal{D}(J)} b_I^2 m_Iv^{-1}\\
                   &\leq &  C[v]^2_{A_2^d}\sum_{I\in\mathcal{D}(J)}\frac{b_I^2}{m_Iv}\\
                   & \leq &  C [v]^2_{A_2^d} u^{-1}(J)
\end{eqnarray*}
Hence we conclude that $\displaystyle{\sup_{I\in\mathcal{D}} \frac{1}{u^{-1}(I)}\int_I|b(x)-m_Ib|^2v^{-1}(x)\,dx <\infty}$.

Similarly if we assume (b),
 we will conclude $\displaystyle{\sup_{I\in\mathcal{D}} \frac{1}{v(I)}\int_I|b(x)-m_Ib|^2u(x)\,dx <\infty}$,
 using this time that $u\in A_2^d$.
These integral conditions are each separately equivalent to $b\in BMO (\mu^{-1})$ when $u,v\in A_2^d$ by the results in \cite[Theorem 4.1]{HoLWic1}.

Assume now that $b\in BMO (\mu^{-1})$ and $u,v\in A_2^d$. We will show that $b\in Carl_{u,v}\cap Carl_{v^{-1},u^{-1}}$. The assumption
implies that 
 \[ \|(b-m_Jb)\mathbbm{1}_J\|^2_{L^2(v^{-1})} \leq Cu^{-1}(J) \;\;\mbox{and}\;\; \|(b-m_Jb)\mathbbm{1}_J\|^2_{L^2(u)} \leq Cv(J).\]
Both $u,v\in A_2^d$ so are $u^{-1}, v^{-1}\in A_2^d$,  also $1\leq m_Iv\,m_Iv^{-1}$, and the dyadic square function is bounded in $L^2(w)$
for $w\in A_2^d$, moreover $\|S^d \big( (b-b_J)\mathbbm{1}_J\big )\|_{L^2(w)}^2= \sum_{I\in\mathcal{D}(J)} |b_I|^2m_Iw$. We therefore conclude that 
\[ \sum_{I\in\mathcal{D}(J)}\frac{|b_I|^2}{m_Iv}\leq \sum_{I\in\mathcal{D}(J)}|b_I|^2m_Iv^{-1}  \leq C[v]^2_{A_2^d} \|(b-m_Jb)\mathbbm{1}_J\|^2_{L^2(v^{-1})}\leq Cu^{-1}(J),\]
\[ \sum_{I\in\mathcal{D}(J)}\frac{|b_I|^2}{m_Iu^{-1}}\leq\sum_{I\in\mathcal{D}(J)}|b_I|^2m_Iu  \leq C[u]^2_{A_2^d} \|(b-m_Jb)\mathbbm{1}_J\|^2_{L^2(u)}\leq Cv(J).\]
Hence $b\in Carl_{u,v}\cup Carl_{v^{-1},u^{-1}}$.

All together we have shown $Carl_{u,v}\cup Carl_{v^{-1},u^{-1}}\subset BMO (\mu^{-1}) \subset Carl_{u,v}\cap Carl_{v^{-1},u^{-1}}$ which implies
that $Carl_{u,v} = BMO (\mu^{-1}) = Carl_{v^{-1},u^{-1}}$ when $u,v\in A_2^d$.
 
\end{proof}

We just showed that when $u,v\in A_2^d$ and the dyadic paraproduct $\pi_b$ is bounded from $L^2(u)$ into $L^2(v)$ then $b\in Carl_{v,u}$. Compare to Corollary~\ref{cor:HoLWic} where only $v\in A_2^d$ and the pair $(u,v)$ is in joint $\mathcal{A}_2$, but we assume $b\in Carl_{u,v}$ (note that the roles of $u$ and $v$ have been interchanged, and in general $Carl_{u,v}\neq Carl_{v,u}$).

When we assume only $v\in A_2^d$ then $\pi_b$ is bounded from $L^2(u)$ into $L^2(v)$ iff \eqref{condition3vA2}, that is
\[ B_2(u,v):=\sup_{I\in\mathcal{D}}\frac{1}{u^{-1}(I)}\sum_{J\in\mathcal{D}(I)} b_J^2m_J(u^{-1})^2m_Jv <\infty.\]

\begin{lemma} If $(u,v)\in \mathcal{A}_2^d$ and $b\in Carl_{u,v}$ with intensity $\mathcal{B}_{u,v}$ then $B_2(u,v)<\infty$. Moreover
\[ B_2(u,v) \leq [u,v]_{\mathcal{A}_2^d}^2 \mathcal{B}_{u,v}.\]
\end{lemma}
\begin{proof} 
The result follows immediately using first the joint $\mathcal{A}_2$ condition and then the $Carl_{u,v}$ condition,
\begin{equation*}
 \sum_{J\in\mathcal{D}(I)} b_J^2m_J(u^{-1})^2m_Jv  \leq  [u,v]_{\mathcal{A}_2^d}^2  \sum_{J\in\mathcal{D}(I)}\frac{b_J^2}{m_Jv} 
                  \leq  [u,v]_{\mathcal{A}_2^d}^2 \mathcal{B}_{u,v}  u^{-1}(I).
\end{equation*}
This implies $B_2(u,v) \leq [u,v]_{\mathcal{A}_2^d}^2 \mathcal{B}_{u,v} <\infty$ as required.
\end{proof}

Using the results in \cite{HoLWic2} we will conclude that
\begin{corollary}\label{cor:HoLWic} If $(u,v)\in \mathcal{A}_2^d$, $v\in A_2^d$, and $b\in Carl_{u,v}$ then $\pi_b$ is bounded
from $L^2(u)$ into $L^2(v)$.
\end{corollary}

 As observed in \cite{M} if $(u,v)\in \mathcal{A}_2^d$, $v\in A_2^d$ (or $u\in A_2^d$), and $b\in BMO^d$ then
the boundedness of the paraproduct reduces to one weight  boundedness on $L^2(v)$ (or on $L^2(u)$). The observation 
being that joint $\mathcal{A}_2$ implies, by the Lebesgue Differentiation Theorem, that $v(x)\leq [u,v]_{\mathcal{A}_2^d} u(x)$ for a.e. $x$, and therefore
$\|g\|_{L^2(v)}\leq [u,v]^{1/2}_{\mathcal{A}_2^d}\|g\|_{L^2(u)}$.
If $v\in A_2^d$ then by Beznosova's one weight linear bound for the paraproduct in $L^2(v)$ \cite{Be1} one has
\[ \|\pi_bf\|_{L^2(v)} \leq C[b]_{BMO^d}[v]_{A_2^d}\|f\|_{L^2(v)}\leq C[b]_{BMO^d}[v]_{A_2^d} [u,v]^{1/2}_{\mathcal{A}_2^d}\|f\|_{L^2(u)}.\]
Likewise if $u\in A_2^d$, then
\[ \|\pi_bf\|_{L^2(v)} \leq [u,v]^{1/2}_{\mathcal{A}_2^d}\|\pi_bf\|_{L^2(u)}\leq C[b]_{BMO^d}[u]_{A_2^d} [u,v]^{1/2}_{\mathcal{A}_2^d}\|f\|_{L^2(u)},\]
where we used Beznosova's result in the last inequality. Using this observation
we can deduce Corollary \ref{cor:HoLWic} without using the machinery of \cite{HoLWic2} 
if we can prove that $(u,v)\in \mathcal{A}_2^d$, $v\in A_2^d$, and $b\in Carl_{u,v}$ imply $b\in BMO^d$.

\begin{lemma} If $(u,v)\in \mathcal{A}_2^d$, $v\in A_2^d$, and $b\in Carl_{u,v}\cap L^2_{loc}(\R )$ then $b\in BMO^d$.
\end{lemma}
\begin{proof} Suffices to show that $b\in Carl_{v,v}$.
Notice that the Cauchy-Schwarz inequality and the joint $\mathcal{A}_2$ condition   imply $\frac{1}{v^{-1}(J)}\leq v(J)\leq \frac{[u,v]_{\mathcal{A}_2^d}}{u^{-1}(J)}$,
therefore
\[ \frac{1}{v^{-1}(J)}\sum_{I\in\mathcal{D}(J)} \frac{b_I^2}{m_Iv}\leq 
   \frac{[u,v]_{\mathcal{A}_2^d}}{u^{-1}(J)} \sum_{I\in\mathcal{D}(J)}\frac{b_I^2}{m_Iv} \leq
     [u,v]_{\mathcal{A}_2^d} \mathcal{B}_{u,v}.\]
We conclude that $b\in Carl_{v,v}\cap L^2_{loc}(\R) = BMO^d$ by Corollary~\ref{C:BMOCvv}.
\end{proof}

It may be worth to point out that when $u=v\in A_2$,   the condition   $\mathcal{B}_2(v,v)<\infty$ coincides with $b\in Carl_{v,v}$.  The reason being that now we do have the lower bound as well as the upper bound  $1\leq m_Iv \,m_I(v^{-1})\leq [v]_{A_2}$.

\begin{lemma} if $w\in A_2$ then $b\in Carl_{w,w}$  if and only if $\mathcal{B}_2(w)<\infty$, where
\[ \mathcal{B}_2(w):= \mathcal{B}_2(w,w)=\sup_{J\in\mathcal{D}} \frac{1}{w^{-1}(J)} \sum_{I\in\mathcal{D}(J)} m_I^2(w^{-1})|b_I|^2m_Iw.\]
\end{lemma}


\section{The maximal and the square functions}\label{Square}

 In this section we relate the boundedness of the square function with the boundedness
of the Maximal function from $L^2(u)$ into $L^2(v)$. The main result of this section states that if the weight $v$ is in $RH_1^d$
and the Maximal function is bounded then the square function is also bounded. This
result is an adaptation of Buckley's proof \cite{Bu}, for the fact that if
$w\in A^d_2$ then $S^d$ is bounded on $L^2(w)$. The last author proved a similar result, in \cite{P}, for the weighted
maximal function and the weighted square function in $L^q(\R)$ and $1<q<\infty$.

\begin{theorem}\label{MaSq} Let $(u,v)$ be a pair of weights such that $v\in RH^d_1$
and the Maximal function $M$ is bounded from $L^2(u)$ into $L^2(v)$ with bound
$\mathcal{M}_{u,v}$ then there exist $C>0$ such that
$$\|S^df\|_{L^2(v)}\leq C\mathcal{M}_{u,v}(1+[v]^{1/2}_{RH_1^d})\|f\|_{L^2(u)}\,.$$
\end{theorem}

As an immediate Corollary of Theorem~\ref{MaSq} and Theorem~\ref{T:PzR} we get,

\begin{corollary}\label{cor:SRH1} Assume $(u,v)\in \mathcal{A}_2^d$, $u^{-1}\in RH_1^d$, and $v\in RH_1^d$, then
$$\|S^df\|_{L^2(v)}\leq C\big ([u,v]_{\mathcal{A}_2^d}[u^{-1}]_{RH_1^d}\big )^{1/2}(1+[v]^{1/2}_{RH_1^d}) \|f\|_{L^2(u)}.$$
\end{corollary}


Note that this estimate does not recover the linear estimate in the one weight case $u=v\in A_2$, it is off by a factor of the form $[v]_{RH_1}^{1/2}$,
unlike the estimate we will present in Theorem~\ref{T:square2l}.

\begin{proof}[Proof of Theorem~\ref{MaSq}]

Given real-valued  $f \in L^2(u)$ we have
\begin{align*}
\|S^df\|^2_{L^2(v)}=&\sum_{I\in\mathcal{D}}|\langle f,h_I\rangle|^2m_Iv
=\frac{1}{2}\sum_{I\in\mathcal{D}}|m_If-m_{\hat{I}}f|^2v(\hat{I}) \nonumber \\
=&\frac{1}{2}\sum_{I\in\mathcal{D}}\big(m_I^2f-m^2_{\hat{I}}f\big)v(\hat{I}):= \Sigma_1\,.
\end{align*}

\noindent Adding and subtracting $2v(I)m_I^2f$ in the sum and rearranging
$$\Sigma_1=\sum_{I\in\mathcal{D}}\big(2v(I)m_I^2f-v(\hat{I})m^2_{\hat{I}}f\big)
+\sum_{I\in\mathcal{D}}\big(v(\hat{I})-2v(I)\big)m^2_If=:\Sigma_2+\Sigma_3\,.$$

\noindent Therefore, it is enough to check that for all $f\in L^2(u)$:
$$|\Sigma_i |\leq C\mathcal{M}_{u,v}^2(1+[v]^{1/2}_{RH_1^d})^2\|f\|_{L^2(u)}^2\, \qquad \text{for} \qquad i=\,2\, ,3.$$

\noindent\textbf{Estimating} $\Sigma_2$: First, let $a_m:=\sum_{I\in\mathcal{D}_m}2v(I)m_I^2f
=2\int(E_mf(x))^2v(x)dx$ where $E_mf(x):=m_If$ for $x\in I\in \mathcal{D}_m$ and $\mathcal{D}_m$ is the collection of all dyadic intervals with length $2^{-m}$. Then
$$\Sigma_2:=\sum_{I\in\mathcal{D}}\big(2v(I)m_I^2f-v(\hat{I})m^2_{\hat{I}}f\big)
=\sum_{m=-\infty}^{\infty}(a_m-a_{m-1})\,.$$
Using the fact that $E_mf(x) \leq Mf(x)$ for all $x \in \R$ we can bound each $a_m$ by
$$|a_m|\leq 2\int_{\R}|Mf(x)|^2v(x)dx=2\|Mf\|_{L^2(v)}^2\leq C\mathcal{M}^2_{u,v} \|f\|^2_{L^2(u)}\,.$$


The last inequality follows since $M$ is assumed to be bounded from $L^2(u)$ to $L^2(v)$.
Let $s_n:=\sum_{|m|\leq n}(a_m-a_{m-1})$, the partial sum sequence of $\Sigma_2$. Since this is a telescoping sum we have $s_n=(a_n-a_{-n-1})$ for all $n \in \nn$. Therefore $|s_n| \leq 2C\mathcal{M}^2_{u,v} \|f\|^2_{L^2(u)}$ for all $n \in \mathbb{N}$ which leads us to the better than desired estimate
$$|\Sigma_{2}|\leq C\mathcal{M}^2_{u,v}\|f\|_{L^2(u)}^2\,.$$

\noindent\textbf{Estimating} $\Sigma_3$:
Since every interval has two children, switching the sum over $I$ to a sum over the parents $J=\hat{I}$ we have the following cancellation,
$$\sum_{I\in\mathcal{D}}\big(v(\hat{I})-2v(I)\big)m^2_{\hat{I}}f= \sum_{J\in\mathcal{D}}\big (v(J)-2v(J_+)+v(J)-2v(J_-)\big )m_J^2f=0\,.$$

\noindent Hence we can write
$$\Sigma_3=\sum_{I\in\mathcal{D}}\big(v(\hat{I})-2v(I)\big)\big(m^2_If-m^2_{\hat{I}}f\big)\,.$$

\noindent Applying the Cauchy-Schwarz inequality,
\begin{align*}
|\Sigma_3 |&\leq\bigg(\sum_{I\in\mathcal{D}}
\frac{\big(v(\hat{I})-2v(I)\big)^2}{v(\hat{I})}(m_If+m_{\hat{I}}f)^2\bigg)^{1/2}
\bigg(\sum_{I\in\mathcal{D}}v(\hat{I})(m_If-m_{\hat{I}}f)^2\bigg)^{1/2}\\
&=\sqrt{\Sigma_4\Sigma_1}\leq\frac{\Sigma_4+\Sigma_1}{2}\,,
\end{align*}
where $\Sigma_4:=\sum_{I\in\mathcal{D}}
\frac{\big(v(\hat{I})-2v(I)\big)^2}{v(\hat{I})}(m_If+m_{\hat{I}}f)^2$.
Thus,
$$\Sigma_1\leq |\Sigma_2|+|\Sigma_3|
\leq C\mathcal{M}^2_{u,v}\|f\|_{L^2(u)}^2+\frac{\Sigma_4+\Sigma_1}{2}\,.$$
Subtracting $\dfrac{\Sigma_1}{2}$ from both sides of this inequality and multiplying by $2$ give us
\begin{equation}
\Sigma_1\leq 2C\mathcal{M}^2_{u,v}\|f\|^2_{L^2(u)}+\Sigma_4\,. \label{E:SqMax}
\end{equation}

\noindent\textbf{Estimating} $\Sigma_4$:
Note that
$m_If\leq 2m_{\hat{I}} f$. Switching the sum over $I$ to a sum over the parents $J=\hat{I}$ gives
$$ \Sigma_4
\leq 3^2\sum_{J\in\mathcal{D}}\frac{|\Delta_J v|^2}{m_J v}\,|J|\,m^2_Jf.$$
Thus
\begin{align*}
\Sigma_4&\lesssim\sum_{I\in\mathcal{D}}\frac{|\Delta_I v|^2}{m_Iv}|I|m^2_If
\lesssim \sum_{I\in\mathcal{D}}\frac{|\Delta_Iv|^2}{m_Iv}|I|\inf_{x\in I}M^2f(x)\\
&\lesssim [v]_{RH_1^d}\int_{\R}M^2f(x)v(x)dx=[v]_{RH_1^d}\|Mf\|^2_{L^2(v)}
\leq [v]_{RH_1^d}\mathcal{M}^2_{u,v}\|f\|^2_{L^2(u)}\,.
\end{align*}

Note that in the third inequality we use the fact that if $v\in RH_1^d$ then,
by Theorem \ref{O:SharpB}, $\{|\Delta_Iv|^2|I|/m_Iv\}_{I\in\mathcal{D}}$ is
a $v$-Carleson sequence with intensity $[v]_{RH_1^d}$. This estimate together with \eqref{E:SqMax} give us the desired estimate for real-valued functions.
Using this estimate for the real and complex parts of $f\in L^2(v)$  we  will conclude that
$$\|S^df\|_{L^2(v)}\leq C\mathcal{M}_{u,v}(1+[v]^{1/2}_{RH_1^d})\|f\|_{L^2(u)}\,.$$
\end{proof}

Even though not explicitly we are still assuming that $(u,v)\in \mathcal{A}_2^d$, since
we assumed that $M:L^2(u)\rightarrow L^2(v)$ which implies $(u,v)\in \mathcal{A}^2_d$, see \cite{GF}.

\begin{remark}
 In the last theorem we are providing a connection between the boundedness of the square function and  the boundedness of the Maximal function. Another novelty of this result is that we have an estimate on how the norm of the square function depends on $[v]_{RH_1^d}$ and the norm of the Maximal function.
\end{remark}

As a consequence of Theorem \ref{MaSq} and Theorem~\ref{T:TWSq},  we can show that the boundeness of the Maximal function from $L^2(u)$ into $L^2(v)$ together with the assumption that $v \in RH_1^d$ will imply the boundedness of the martingale transform.

\begin{theorem}\label{C:martingale}
Let $(u,v)$ be a pair of weights such that $v\in RH_1^d$ and the Maximal function
$M$ is bounded from $L^2(u)$ into $L^2(v)$ then the martingale transforms  $T_r$ are
uniformly bounded from $L^2(u)$ into $L^2(v)\,.$
\end{theorem}


\begin{proof}
Let us consider a pair of weights $(u,v)$ satisfying the assumptions. By
Theorem \ref{MaSq} the dyadic square function is bounded, and by Theorem \ref{T:TWSq} the pair of weights $(u,v)$  satisfies

\begin{itemize}
\item[(i)] $(u,v)\in \mathcal{A}^d_2$
\item[(ii)] $\{|\Delta_Iu^{-1}|^2m_Iv|I|\}_{I\in\mathcal{D}}$ is a $u^{-1}$-Carleson sequence.
\end{itemize}

 Let us denote the intensity of the $u^{-1}$-Carleson sequence in (ii)
by $\mathcal{D}_{u,v}$. To prove the boundedness of the martingale transform $T_r$,
we need to show that $(u,v)$ also satisfies the last two conditions in Theorem~\ref{T:TWSq}

\begin{itemize}
\item[(iii)] $\{|\Delta_Iv|^2m_I(u^{-1})|I|\}_{I\in\mathcal{D}}$ is a $v$-Carleson sequence.
\item[(iv)] The operator $T_0$ is bounded from $L^2(u)$ into $L^2(v)$.
\end{itemize}

 For condition (iii), we use the assumption
 $v\in RH_1^d$, Theorem \ref{O:SharpB}(b), and $(u,v)\in\mathcal{A}_2^d$. More precisely, for any $J\in\mathcal{D}$,
\begin{align*}
\sum_{I\in\mathcal{D}}|\Delta_Iv|^2m_I(u^{-1})|I|&=\sum_{I\in\mathcal{D}(J)}
\frac{|\Delta_Iv|^2}{m_Iv}m_Iv\,m_I(u^{-1})|I|\\
&\leq [u,v]_{\mathcal{A}_2^d}\sum_{I\in\mathcal{D}(J)}\frac{|\Delta_Iv|^2}{m_Iv}|I|\leq C[u,v]_{\mathcal{A}_2^d}[v]_{RH^d_1}v(J)\,.
\end{align*}

 We now need to check condition (iv), which for
any positive $f\in L^2(u^{-1})$ and $g\in L^2(v)$  is equivalent to
$$|\langle T_0(fu^{-1}),gv\rangle|\leq C\|f\|_{L^2(u^{-1})}\|g\|_{L^2(v)}\, .$$
Thus, it suffices to verify the estimate
\begin{equation}
\sum_{I\in\mathcal{D}}m_I(|f|u^{-1})m_I(|g|v)\frac{|\Delta_Iv|}{m_Iv}
\frac{|\Delta_Iu^{-1}|}{m_I(u^{-1})}|I|\leq C\|f\|_{L^2(u^{-1})}\|g\|_{L^2(v)}\,.\label{DualT0}
\end{equation}

 To see that (\ref{DualT0}) holds, we use the Cauchy-Schwarz inequality:
\begin{align*}
\sum_{I\in\mathcal{D}}&m_I(|f|u^{-1})m_I(|g|v)\frac{|\Delta_Iv|}{m_Iv}
    \frac{|\Delta_Iu^{-1}|}{m_I(u^{-1})}|I|\\
&\leq\bigg(\sum_{I\in\mathcal{D}}\bigg(\frac{m_I(|f|u^{-1})}{m_I(u^{-1})}\bigg)^2
    |\Delta_Iu^{-1}|^2m_Iv|I|\bigg)^{1/2}
   \bigg(\sum_{I\in\mathcal{D}}\bigg(\frac{m_I(|g|v)}{m_Iv}\bigg)^2\frac{|\Delta_Iv|^2}{m_Iv}|I|\bigg)^{1/2}\\
&\leq\bigg(\sum_{I\in\mathcal{D}}|\Delta_Iu^{-1}|^2m_Iv|I|\inf_{x\in I}M^2_{u^{-1}}f(x)\bigg)^{1/2}
   \bigg(\sum_{I\in\mathcal{D}}\frac{|\Delta_Iv|^2}{m_Iv}|I|\inf_{x\in I}M^2_vg(x)\bigg)^{1/2}\,.
\end{align*}

Since $|\Delta_Iu^{-1}|^2m_Iv|I|$ is a $u^{-1}$-Carleson sequence with intensity $\mathcal{D}_{u,v}$ and
$\frac{|\Delta_Iv|^2}{m_Iv}|I|$ is a $v$-Carleson sequence with intensity $[v]_{RH_1^d}$, by condition (ii) and Theorem \ref{O:SharpB}(b) respectively, we have that
\begin{align*}
\sum_{I\in\mathcal{D}}m_I(|f|u^{-1})m_I(|g|v)\frac{|\Delta_Iv|}{m_Iv}\frac{|\Delta_Iu^{-1}|}{m_I(u^{-1})}|I|
&\leq \sqrt{\mathcal{D}_{u,v}[v]_{RH_1^d}}\|M_{u^{-1}}f\|_{L^2(u^{-1})}\|M_vg\|_{L^2(v)}\\
&\leq 8\sqrt{\mathcal{D}_{u,v}[v]_{RH_1^d}}\|f\|_{L^2(u^{-1})}\|g\|_{L^2(v)}\,,
\end{align*}
the last inequality by Theorem~\ref{T:WMF}.
\end{proof}

As an immediate consequence of Theorem~\ref{C:martingale} and Corollary~\ref{cor:SRH1} we get the following corollary.

\begin{corollary} If $(u,v)\in\mathcal{A}_2^d$, $u^{-1}\in RH_1^d$ and $v\in RH_1^d$ then the martingale transforms $T_r$ are
uniformly bounded  $L^2(u)$ into $L^2(v)$.
\end{corollary}


\section{The sharp quantitative estimate for the dyadic square function} \label{sec:square2}

Our last theorem provides the dependence of the operator norm $\|S^d\|_{L^2(u)\rightarrow L^2(v)}$ on the
joint $\mathcal{A}_2$ characteristic of the weights and $[u^{-1}]_{RH_1^d}$. This extends results of Beznosova \cite{Be2}, and we follow the template of her original proof.

\begin{theorem}\label{T:square2l} Let $(u,v)$ be a pair of weights such that $(u,v)\in \mathcal{A}_2^d$ and
$u^{-1}\in RH_1^d\,.$ Then there is a constant such that
$$\|S^d\|_{L^2(u)\rightarrow L^2(v)}
\leq C[u,v]_{\mathcal{A}^d_{2}}^{1/2}\big (1+[u^{-1}]_{RH_1^d}^{1/2}\big )\,.$$
\end{theorem}

\begin{proof}
We can write the square of the norm $\|S^df\|_{L^2(v)}$ as:
\begin{align*}
\|S^df\|_{L^2(v)}^2&=\int\sum_{I\in\mathcal{D}}|\langle f,h_I\rangle |^2
    \frac{\mathbbm{1}_I(x)}{|I|}\,v(x)dx=\sum_{I\in\mathcal{D}}|\langle f,h_I\rangle |^2m_Iv.
\end{align*}

We decompose $h_I$ in a  slightly different way.
For any weight $u^{-1}$, we can write $h_I$ as
$$h_I(x)=\frac{1}{\sqrt{|I|}}\Big(H^{u^{-1}}_I(x)+A_I^{u^{-1}}\mathbbm{1}_I(x)\Big)
\quad\textrm{where}\quad A_I^{u^{-1}}=\frac{\Delta_Iu^{-1}}{2m_I(u^{-1})}\,.$$

The family $\{u^{-1/2}H_I^{u^{-1}}\}_{I\in\mathcal{D}}$ is orthogonal in $L^2(dx)$ with norms satisfying the inequality
$$\|u^{-1/2}H^{u^{-1}}_I\|_{L^2(\R)}\leq \sqrt{|I|m_I(u^{-1})}.$$
Hence
by Bessel's inequality we have that for all
$f\in L^2(u)$ (recall that $f\in L^2(u)$ if and only if $fu^{1/2}\in L^2(\R)$),
$$\sum_{I\in\mathcal{D}}\frac{|\langle f,H_I^{u^{-1}}\rangle |^2}{|I|m_I(u^{-1})}
\leq \|fu^{1/2}\|_{L^2(\R)}^2=\|f\|^2_{L^2(u)}\,. $$ 
Since $m_Iv \leq [u,v]_{\mathcal{A}_2^d}/m_I(u^{-1})$ we conclude that for all $f\in L^2(u)$,
\begin{equation}
\sum_{I\in\mathcal{D}}\bigg |\bigg \langle f,\frac{H^{u^{-1}}_I}{\sqrt{|I|}}\bigg\rangle\bigg |^2m_Iv
\leq [u,v]_{\mathcal{A}_2^d}\|f\|^2_{L^2(u)}\, .\label{Sqf:E1}
\end{equation}
We claim that
\begin{equation}
\sum_{I\in\mathcal{D}}\bigg |\bigg \langle f,\frac{A_I^{u^{-1}}\mathbbm{1}_I}{\sqrt{|I|}}\bigg \rangle\bigg |^2m_Iv
\leq C[u,v]_{\mathcal{A}_2^d}[u^{-1}]_{RH^1_d}\|f\|^2_{L^2(u)}\,.\label{Sqf:E2}
\end{equation}

Using estimates (\ref{Sqf:E1}) and (\ref{Sqf:E2}) and the Cauchy-Schwarz inequality we conclude that
$$\sum_{I\in\mathcal{D}}|\langle f,h_I\rangle |^2\, m_Iv
\leq C\big([u,v]_{\mathcal{A}_2^d}+2[u,v]_{A^2_d}[u^{-1}]^{1/2}_{RH_1^d}+ [u,v]_{\mathcal{A}^2_d}[u^{-1}]_{RH^d_1}\big)\|f\|^2_{L^2(u)}\,,$$
 which completes the proof.

 Let us return to our claim.
 The left hand side of (\ref{Sqf:E2}) can be written as
$$\sum_{I\in\mathcal{D}}\bigg |\bigg\langle f,\frac{A_I^{u^{-1}}\mathbbm{1}_I}{\sqrt{|I|}}\bigg\rangle\bigg |^2m_Iv
=\frac{1}{4}\sum_{I\in\mathcal{D}}|m_If |^2\bigg(\frac{\Delta_I u^{-1}}{m_I(u^{-1})}\bigg)^2|I|m_Iv\,.$$

 By our assumptions: $(u,v)\in\mathcal{A}_2^d$ and $u^{-1}\in RH_1^d$,
for any $J\in\mathcal{D}$, we have
\begin{eqnarray}
\frac{1}{|J|}\sum_{I\in\mathcal{D}(J)}\bigg(\frac{\Delta_I u^{-1}}{m_I(u^{-1})}\bigg)^2|I|m^2_I(u^{-1})m_Iv 
&\leq & \frac{[u,v]_{A^d_2}}{|J|}\sum_{I\in\mathcal{D}(J)}\bigg(\frac{\Delta_I u^{-1}}{m_I(u^{-1})}\bigg)^2|I|m_I(u^{-1})\nonumber\\
&\leq &[u,v]_{\mathcal{A}_2^d}[u^{-1}]_{RH^d_1}m_J(u^{-1})\label{S:Buckley}\,,
\end{eqnarray}

The last inequality (\ref{S:Buckley}) is an  application of Lemma \ref{O:SharpB}(b). Therefore the sequence
$\{\alpha_I:=\big(\frac{\Delta_Iu^{-1}}{m_I(u^{-1})}\big)^2|I|m^2_I(u^{-1})m_Iv\}_{I\in\mathcal{D}}$ is a
 $u^{-1}$-Carleson sequence with intensity $[u,v]_{A_2^d}[u^{-1}]_{RH^d_1}\,.$

 We now  can prove the claimed estimate (\ref{Sqf:E2}) ,
\begin{align*}
\sum_{I\in\mathcal{D}}|m_If|^2\bigg(\frac{\Delta_I u^{-1}}{m_I(u^{-1})}\bigg)^2|I|m_Iv
&=\sum_{I\in\mathcal{D}}\bigg(\frac{|m_If|}{m_I(u^{-1})}\bigg)^2\bigg(\frac{\Delta_I u^{-1}}{m_I(u^{-1})}\bigg)^2|I|m^2_I(u^{-1})m_Iv\\
&\leq \sum_{I\in\mathcal{D}}\Big(m_I^{u^{-1}}(|f|u)\Big)^2\bigg(\frac{\Delta_I u^{-1}}{m_I(u^{-1})}\bigg)^2|I|m^2_I(u^{-1})m_Iv\\
&\leq \sum_{I\in\mathcal{D}}\inf_{x\in I}M^2_{u^{-1}}(fu)(x)\bigg(\frac{\Delta_I u^{-1}}{m_I(u^{-1})}\bigg)^2|I|m^2_I(u^{-1})m_Iv.
\end{align*}
Finally using Lemma~\ref{WCL} \ with $F(x)= M^2_{u^{-1}}(fu)(x)$ and the $u^{-1}$-Carleson sequence $\{\alpha_{I}\}_{I\in\mathcal{D}}$
with intensity $[u,v]_{A_2^d}[u^{-1}]_{RH^d_1}\,,$
will give us that
\begin{align*}
\sum_{I\in\mathcal{D}}|m_If|^2\bigg(\frac{\Delta_I u^{-1}}{m_I(u^{-1})}\bigg)^2|I|m_Iv
&\leq C[u,v]_{\mathcal{A}^d_2}[u^{-1}]_{RH^d_1}\|M_{u^{-1}}(fu)\|_{L^2(u^{-1})}^2\\
&\leq C[u,v]_{\mathcal{A}^d_2}[u^{-1}]_{RH^d_1}\|fu\|_{L^2(u^{-1})}^2 \\
&= C[u,v]_{\mathcal{A}^d_2}[u^{-1}]_{RH^d_1}\|f\|_{L^2(u)}^2\,.
\end{align*}
\end{proof}

Analyzing carefully the proof above we realize that if instead of assuming $u^{-1}\in RH_1^d$ we assume
that   $\{ | \Delta_I u^{-1} |^2 m_Iv$ is a $u^{-1}$-Carleson sequence with intensity $\mathcal{C}_{u,v}$
the argument will go through and we will recover the  Lacey-Li estimate. 

\begin{theorem} Let $(u,v)$ be a pair of weights such that $(u,v)\in \mathcal{A}_2^d$ and
$\{ | \Delta_I u^{-1} |^2 m_Iv\}_{I\in\mathcal{D}}$ is a $u^{-1}$-Carleson sequence with intensity $\mathcal{C}_{u,v}$. Then there is a constant $C>0$ such that
$$\|S^d\|_{L^2(u)\rightarrow L^2(v)}\leq C \big ([u,v]_{\mathcal{A}^d_{2}} +\mathcal{C}_{u,v}\big )^{1/2} \,.$$
\end{theorem}

We leave the details of the proof to the reader.



\end{document}